\documentclass[11pt]{amsart}
\usepackage[utf8]{inputenc}
\usepackage{kantlipsum}
\usepackage{comment}
\calclayout
\newif\ifdraft
\draftfalse
\usepackage{amsmath,amsfonts,amsthm,mathrsfs}
\usepackage{amssymb, stmaryrd}

\usepackage[unicode,bookmarks]{hyperref}
\usepackage[usenames,dvipsnames]{xcolor}
\hypersetup{colorlinks=true,citecolor=NavyBlue,linkcolor=Brown,urlcolor=Orange}

\usepackage[alphabetic,initials]{amsrefs}

\usepackage{enumitem}

\usepackage{chngcntr}

\ifdraft
\usepackage[notcite,notref,color]{showkeys}

\definecolor{labelkey}{gray}{0.5}
\fi

\usepackage{tikz}

\usepackage{tikz-cd}

\usetikzlibrary{matrix,arrows}
\newlength{\myarrowsize} 

\pgfarrowsdeclare{cmto}{cmto}{
	\pgfsetdash{}{0pt} 
	\pgfsetbeveljoin 
	\pgfsetroundcap 
	\setlength{\myarrowsize}{0.6pt}
	\addtolength{\myarrowsize}{.5\pgflinewidth}
	\pgfarrowsleftextend{-4\myarrowsize-.5\pgflinewidth} 
	\pgfarrowsrightextend{.8\pgflinewidth}
}{
	\setlength{\myarrowsize}{0.6pt} 
  	\addtolength{\myarrowsize}{.5\pgflinewidth}  
	\pgfsetlinewidth{0.5\pgflinewidth}
	\pgfsetroundjoin
	\pgfpathmoveto{\pgfpoint{1.5\pgflinewidth}{0}}
	\pgfpatharc{-109}{-170}{4\myarrowsize}
	\pgfpatharc{10}{189}{0.58\pgflinewidth and 0.2\pgflinewidth}
	\pgfpatharc{-170}{-115}{4\myarrowsize+\pgflinewidth}
	\pgfpathclose
	\pgfusepathqfillstroke
	\pgfpathmoveto{\pgfpoint{1.5\pgflinewidth}{0}}
	\pgfpatharc{109}{170}{4\myarrowsize}
	\pgfpatharc{-10}{-189}{0.58\pgflinewidth and 0.2\pgflinewidth}
	\pgfpatharc{170}{115}{4\myarrowsize+\pgflinewidth}
	\pgfpathclose
	\pgfusepathqfillstroke
	\pgfsetlinewidth{2\pgflinewidth}
}

\pgfarrowsdeclare{cmonto}{cmonto}{
	\pgfsetdash{}{0pt} 
	\pgfsetbeveljoin 
	\pgfsetroundcap 
	\setlength{\myarrowsize}{0.6pt}
	\addtolength{\myarrowsize}{.5\pgflinewidth}
	\pgfarrowsleftextend{-4\myarrowsize-.5\pgflinewidth} 
	\pgfarrowsrightextend{.8\pgflinewidth}
}{
	\setlength{\myarrowsize}{0.6pt} 
  	\addtolength{\myarrowsize}{.5\pgflinewidth}  
	\pgfsetlinewidth{0.5\pgflinewidth}
	\pgfsetroundjoin
	\pgfpathmoveto{\pgfpoint{1.5\pgflinewidth}{0}}
	\pgfpatharc{-109}{-170}{4\myarrowsize}
	\pgfpatharc{10}{189}{0.58\pgflinewidth and 0.2\pgflinewidth}
	\pgfpatharc{-170}{-115}{4\myarrowsize+\pgflinewidth}
	\pgfpathclose
	\pgfusepathqfillstroke
	\pgfpathmoveto{\pgfpoint{1.5\pgflinewidth}{0}}
	\pgfpatharc{109}{170}{4\myarrowsize}
	\pgfpatharc{-10}{-189}{0.58\pgflinewidth and 0.2\pgflinewidth}
	\pgfpatharc{170}{115}{4\myarrowsize+\pgflinewidth}
	\pgfpathclose
	\pgfusepathqfillstroke
	\pgfpathmoveto{\pgfpoint{1.5\pgflinewidth-0.3em}{0}}
	\pgfpatharc{-109}{-170}{4\myarrowsize}
	\pgfpatharc{10}{189}{0.58\pgflinewidth and 0.2\pgflinewidth}
	\pgfpatharc{-170}{-115}{4\myarrowsize+\pgflinewidth}
	\pgfpathclose
	\pgfusepathqfillstroke
	\pgfpathmoveto{\pgfpoint{1.5\pgflinewidth-0.3em}{0}}
	\pgfpatharc{109}{170}{4\myarrowsize}
	\pgfpatharc{-10}{-189}{0.58\pgflinewidth and 0.2\pgflinewidth}
	\pgfpatharc{170}{115}{4\myarrowsize+\pgflinewidth}
	\pgfpathclose
	\pgfusepathqfillstroke
	\pgfsetlinewidth{2\pgflinewidth}
}

\pgfarrowsdeclare{cmhook}{cmhook}{
	\pgfsetdash{}{0pt} 
	\pgfsetbeveljoin 
	\pgfsetroundcap 
	\setlength{\myarrowsize}{0.6pt}
	\addtolength{\myarrowsize}{.5\pgflinewidth}
	\pgfarrowsleftextend{-4\myarrowsize-.5\pgflinewidth} 
	\pgfarrowsrightextend{.8\pgflinewidth}
}{
	\setlength{\myarrowsize}{0.6pt} 
  	\addtolength{\myarrowsize}{.5\pgflinewidth}  
 	\pgfsetdash{}{0pt}
	\pgfsetroundcap
	\pgfpathmoveto{\pgfqpoint{0pt}{-4.667\pgflinewidth}}
	\pgfpathcurveto
    {\pgfqpoint{4\pgflinewidth}{-4.667\pgflinewidth}}
    {\pgfqpoint{4\pgflinewidth}{0pt}}
    {\pgfpointorigin}
	\pgfusepathqstroke
}

\newenvironment{diagram*}[2]{%
\[%
\begin{tikzpicture}[>=cmto,baseline=(current bounding box.center),%
	to/.style={->,font=\scriptsize,cap=round},%
	into/.style={cmhook->,font=\scriptsize,cap=round},%
	onto/.style={-cmonto,font=\scriptsize,cap=round},%
	math/.style={matrix of math nodes, row sep=#2, column sep=#1,%
		text height=1.5ex, text depth=0.25ex}]%
}{%
\end{tikzpicture}%
\]%
\ignorespacesafterend%
}

\newcommand{\cohH}{\mathcal{H}}

\def\overbar#1#2#3{{%
	\setbox0=\hbox{\displaystyle{#1}}%
	\dimen0=\wd0
	\advance\dimen0 by -#2 
	\vbox {\nointerlineskip \moveright #3 \vbox{\hrule height 0.3pt width \dimen0}%
		\nointerlineskip \vskip 1.5pt \box0}%
}}

\makeatletter
\let\@@seccntformat\@seccntformat
\renewcommand*{\@seccntformat}[1]{%
  \expandafter\ifx\csname @seccntformat@#1\endcsname\relax
    \expandafter\@@seccntformat
  \else
    \expandafter
      \csname @seccntformat@#1\expandafter\endcsname
  \fi
    {#1}%
}
\newcommand*{\@seccntformat@subsection}[1]{%
  \textbf{\csname the#1\endcsname.}
}
\makeatother

\makeatletter
\let\@paragraph\paragraph
\renewcommand*{\paragraph}[1]{%
	\vspace{0.3\baselineskip}%
	\@paragraph{\textit{#1}}%
}
\makeatother

\counterwithin{equation}{subsection}
\counterwithout{subsection}{section}
\counterwithin{figure}{subsection}

\newtheorem{theorem}[equation]{Theorem}
\newtheorem*{theorem*}{Theorem}
\newtheorem{lemma}[equation]{Lemma}
\newtheorem*{lemma*}{Lemma}
\newtheorem{corollary}[equation]{Corollary}
\newtheorem{proposition}[equation]{Proposition}
\newtheorem*{proposition*}{Proposition}

\newtheorem{claim}[equation]{Claim}
\theoremstyle{definition}
\newtheorem{definition}[equation]{Definition}
\newtheorem*{definition*}{Definition}
\newtheorem{remark}[equation]{Remark}

\newtheorem{example}[equation]{Example}
\newtheorem*{example*}{Example}
\newtheorem*{problem*}{Problem}

\theoremstyle{plain}

\newcommand{\theoremref}[1]{\hyperref[#1]{Theorem~\ref*{#1}}}
\newcommand{\lemmaref}[1]{\hyperref[#1]{Lemma~\ref*{#1}}}
\newcommand{\definitionref}[1]{\hyperref[#1]{Definition~\ref*{#1}}}
\newcommand{\propositionref}[1]{\hyperref[#1]{Proposition~\ref*{#1}}}
\newcommand{\conjectureref}[1]{\hyperref[#1]{Conjecture~\ref*{#1}}}
\newcommand{\corollaryref}[1]{\hyperref[#1]{Corollary~\ref*{#1}}}
\newcommand{\exampleref}[1]{\hyperref[#1]{Example~\ref*{#1}}}
\newcommand{\setupref}[1]{\hyperref[#1]{Set-up~\ref*{#1}}}
\newcommand{\remarkref}[1]{\hyperref[#1]{Remark~\ref*{#1}}}
\newcommand{\claimref}[1]{\hyperref[#1]{Claim~\ref*{#1}}}
\newcommand{\figureref}[1]{\hyperref[#1]{Figure~\ref*{#1}}}

\makeatletter
\let\old@caption\caption
\renewcommand*{\caption}[1]{%
	\setcounter{figure}{\value{equation}}%
	\stepcounter{equation}%
	\old@caption{#1}\relax%
}
\makeatother

\newcounter{intro}

\newtheorem{intro-conjecture}[intro]{Conjecture}
\newtheorem{intro-corollary}[intro]{Corollary}
\newtheorem{intro-theorem}[intro]{Theorem}

\newcommand{\parref}[1]{\hyperref[#1]{\S\ref*{#1}}}

\makeatletter
\newcommand*\if@single[3]{%
  \setbox0\hbox{${\mathaccent"0362{#1}}^H$}%
  \setbox2\hbox{${\mathaccent"0362{\kern0pt#1}}^H$}%
  \ifdim\ht0=\ht2 #3\else #2\fi
  }
\newcommand*\rel@kern[1]{\kern#1\dimexpr\macc@kerna}
\newcommand*\widebar[1]{\@ifnextchar^{{\wide@bar{#1}{0}}}{\wide@bar{#1}{1}}}
\newcommand*\wide@bar[2]{\if@single{#1}{\wide@bar@{#1}{#2}{1}}{\wide@bar@{#1}{#2}{2}}}
\newcommand*\wide@bar@[3]{%
  \begingroup
  \def\mathaccent##1##2{%
    \if#32 \let\macc@nucleus\first@char \fi
    \setbox\z@\hbox{$\macc@style{\macc@nucleus}_{}$}%
    \setbox\tw@\hbox{$\macc@style{\macc@nucleus}{}_{}$}%
    \dimen@\wd\tw@
    \advance\dimen@-\wd\z@
    \divide\dimen@ 3
    \@tempdima\wd\tw@
    \advance\@tempdima-\scriptspace
    \divide\@tempdima 10
    \advance\dimen@-\@tempdima
    \ifdim\dimen@>\z@ \dimen@0pt\fi
    \rel@kern{0.6}\kern-\dimen@
    \if#31
      \overline{\rel@kern{-0.6}\kern\dimen@\macc@nucleus\rel@kern{0.4}\kern\dimen@}%
      \advance\dimen@0.4\dimexpr\macc@kerna
      \let\final@kern#2%
      \ifdim\dimen@<\z@ \let\final@kern1\fi
      \if\final@kern1 \kern-\dimen@\fi
    \else
      \overline{\rel@kern{-0.6}\kern\dimen@#1}%
    \fi
  }%
  \macc@depth\@ne
  \let\math@bgroup\@empty \let\math@egroup\macc@set@skewchar
  \mathsurround\z@ \frozen@everymath{\mathgroup\macc@group\relax}%
  \macc@set@skewchar\relax
  \let\mathaccentV\macc@nested@a
  \if#31
    \macc@nested@a\relax111{#1}%
  \else
    \def\gobble@till@marker##1\endmarker{}%
    \futurelet\first@char\gobble@till@marker#1\endmarker
    \ifcat\noexpand\first@char A\else
      \def\first@char{}%
    \fi
    \macc@nested@a\relax111{\first@char}%
  \fi
  \endgroup
}
\makeatother

\DeclareMathOperator{\Gr}{Gr}
\newtheorem{set-up}[equation]{Set-up}

\usepackage{tikz-3dplot}
\definecolor{gray(x11gray)}{rgb}{0.75, 0.75, 0.75}
\definecolor{aliceblue}{rgb}{0.94, 0.97, 1.0}
\usepackage{float}
\usepackage[normalem]{ulem}

\begin{document}

\vspace{\baselineskip}

\title[Singularities of secant varieties from a Hodge theoretic perspective]{Singularities of secant varieties from a Hodge theoretic perspective}

\author[S.~Olano]{Sebasti\'an~Olano}
\address{Department of Mathematics, University of Toronto, 40 St. George St., Toronto, Ontario 
Canada, M5S 2E4}
\email{{\tt seolano@math.toronto.edu}}

\author[D.~Raychaudhury]{Debaditya Raychaudhury}
\address{Department of Mathematics, University of Arizona, 617 N. Santa Rita Ave., Tucson, Arizona 85721,
USA}
\email{draychaudhury@math.arizona.edu}

\author[L.~Song]{Lei Song}
\address{School of Mathematics, Sun Yat-sen University, No.~135 Xingang Xi Road, Guangzhou, Guangdong 510275, P. R. China}
\email{songlei3@mail.sysu.edu.cn}

\thanks{}


\maketitle

\vspace{-20pt}

\begin{abstract}
    We study the singularities of secant varieties of smooth projective varieties using methods from birational geometry when the embedding line bundle is sufficiently positive. More precisely, we study the Du Bois complex of secant varieties and its relationship with the sheaves of \textcolor{black}{differential} forms. Through this analysis, we give a necessary and sufficient condition for these varieties to have $p$-Du Bois singularities (in a sense that was proposed in \cite{SVV}). In addition, we show that the singularities of these varieties are never higher rational, by giving a classification of the cases when they are pre-$1$-rational. From these results, we deduce several consequences, including a Kodaira-Akizuki-Nakano type vanishing result for the reflexive differential forms of the secant varieties. 
\end{abstract}

\section{Introduction}

Secant varieties have been vastly studied in the literature. In particular, there has been a great deal of interest in understanding their defining equations and syzygies as well as their singularities (\cite{ENP, Rai, SV1, SV, Ver, V2, V1} and the references therein). The research on these varieties is also partly motivated by topics in algebraic statistics and algebraic complexity theory (\cite{SS, LW}). A few years ago, under some positivity conditions on the embedding line bundle, \cite{Ull16} had shown the normality of these varieties, thereby completing some results of Vermeire.
More recently, the singularities of the secant varieties were shown to be Du Bois, which is an important class of singularities (see \cite{KS2} for a survey), in \cite{CS18}, and also a characterization for when these are rational was given.

\vspace{5pt}

Very recently, the notions of Du Bois and rational singularities have been extended substantially in a series of papers \cite{MPOW, friedmanlaza22, jksy22, KL, SVV}\textcolor{black}{, and in particular, the notions of higher Du Bois and higher rational singularities have emerged}. For this reason, it is natural to ask: {\it to which of these newly defined singularity classes do the secant varieties belong?} \textcolor{black}{Although these generalized notions of $p$-Du Bois singularities are defined via some conditions imposed on the first $p$ associated graded pieces of the Du Bois complex and require $p$ to be in an admissible range (based on the codimension of the singular locus of the variety under consideration), our study of secant varieties in this article will often go beyond that, and we will show a similar behavior on the rest of the graded pieces.}

\vspace{5pt}

To set up the context, we start with a smooth projective variety $X$ of dimension $n$ and a very ample line bundle $L$ on $X$. We denote by $\Sigma=\Sigma(X,L)$ the secant variety of $X$ with respect to $L$. Under a mild assumption on the positivity of $L$, which requires $L$ to be $3$-very ample (see \definitionref{defva}), $\dim \Sigma=2n+1$ and the singular locus of $\Sigma$ is $X\subset \Sigma$ (see \propositionref{stronglog}). 

\vspace{5pt}

In \cite{Ull16} and \cite{CS18}, the singularities of the secant varieties were studied when $X$ is embedded by a sufficiently positive linear series. In this article, the positivity of $L$ is measured through the integer $p$ for which it satisfies the $(Q_p)$-property, see \definitionref{qp}. This property, for a 3-very ample line bundle $L$, is defined through properties $(Q1)$, $(Q2)$ (which have no bearing on $p$) and $(Q3_p)$. It is straightforward to see from the definition that if $L$ satisfies $(Q_p)$-property (more precisely $(Q3_p)$) for some $p\geq 0$, then it satisfies $(Q_k)$-property (more precisely $(Q3_k)$) for all $0\leq k\leq p$. Also, if $L$ satisfies $(Q_n)$-property (resp. $(Q3_n)$), then it satisfies $(Q_p)$-property (resp. $(Q3_p)$) for all $p\geq 0$.

\vspace{5pt}

It turns out that $L$ satisfies $(Q_0)$-property if one of the following holds:
\begin{itemize}
    \item $n=1$ and $\textrm{deg}(L)\geq 2g+3$ where $g$ is the genus of $X$ (in fact, in this case $L$ satisfies $(Q_p)$-property for all $p\geq 0$); or 
    \item $n\geq 2$ and $L=K_X+dA+B$ with $d\geq 2n+2$ where $A$ and $B$ are very ample and nef line bundles respectively. 
\end{itemize}
The main results of Ullery and Chou--Song can be stated as follows:

\begin{theorem}[\cite{Ull16, CS18}] 
Assume $L$ satisfies $(Q_0)$-property. Then:
    \begin{itemize}
        \item[(1)] $\Sigma$ is normal, and its singularities are Du Bois.
        \item[(2)] $\Sigma$ has rational singularities if and only if $H^i(\mathcal{O}_X)=0$ for all $i>0$.\footnote{More precisely, assuming $L$ is 3-very ample, \cite{Ull16} showed the normality of $\Sigma$ when $L$ satisfies $(Q1)$, and \cite{CS18} showed that $\Sigma$ has Du Bois singularities and the assertion (2) when $L$ satisfies $(Q1)$ and $(Q3_0)$.}
    \end{itemize}
\end{theorem}

The present work is devoted to understanding the singularities of $\Sigma$, where we assume that $L$ satisfies $(Q_p)$-property for some $p\geq 0$. Roughly speaking, the larger $p$ for which $L$ satisfies $(Q_p)$-property implies the better singularities of $\Sigma$.

\vspace{5pt}

As for concrete examples of line bundles $L$ that satisfy this property, it turns out that there are explicit functions $f(p,n)$ and $g(l,p,n)$ such that the pluri-adjoint linear series $lK_X+dA+B$ where $A$ is a very ample line bundle and $B$ is a nef line bundle satisfies $(Q_p)$-property if  $l\geq f(p,n)$ and $d\geq g(l,p,n)$, see \theoremref{pluriqp}. In particular, $lK_X+dA+B$ satisfies $(Q_p)$-property for all $p\geq 0$ if $d\gg l\gg 0$. For the convenience of the reader, here we point out special cases of \theoremref{pluriqp}: for $n\geq 2$, the line bundle 
$$L = 2K_X + dA + B$$ satisfies $(Q_1)$-property, when $d\geq 3n +4$ (if $n\geq 3$, then $d = 3n+3$ also satisfies the condition). 
Moreover, for $n\geq 2$ and $1\leq p\leq n$, the line bundle
\small
$$L=\begin{cases} 
      \left[\binom{n-1}{p-1}+1\right]K_X+\left[\left(\binom{n-1}{p-1}+1\right)(n+2)+2p\left(\binom{n}{\lfloor\frac{n}{2}\rfloor}+1\right)\right]A+B & \textrm{if } p-1\leq \lfloor\binom{n-1}{2}\rfloor; \\[5pt]
      \left[\binom{n-1}{\lfloor\frac{n-1}{2}\rfloor}+1\right]K_X+\left[\left(\binom{n-1}{\lfloor\frac{n-1}{2}\rfloor}+1\right)(n+2)+2p\left(\binom{n}{\lfloor\frac{n}{2}\rfloor}+1\right)\right]A+B & \textrm{otherwise}
   \end{cases}$$
\normalsize
satisfies $(Q_p)$-property.\\

\noindent\textbf{Higher Du Bois singularities.} We now recall that associated to any complex variety $Z$ there is the Du Bois complex $\underline{\Omega}^{\bullet}_Z$, introduced in \cite{dubois81}, which is an object in the derived category of filtered complexes on $Z$. The associated graded objects $$\underline{\Omega}^{k}_Z := \Gr_F^k\underline{\Omega}^{\bullet}_Z[k]$$ are objects in the derived category of coherent sheaves. A variety $Z$ is said to have Du Bois singularities if the canonical morphism $\mathcal{O}_Z \to \underline{\Omega}^{0}_Z$ is a quasi-isomorphism. This notion has been generalized by requiring the canonical morphisms $\Omega^k_Z \to \underline{\Omega}^{k}_Z$ to be quasi-isomorphisms for $0\leq k\leq p$ as well, and varieties that satisfy this property are said to have $p$-Du Bois singularities. The condition is well-behaved for varieties whose singularities are local complete intersections (see \cite{MP,MP22}) but usually fails otherwise, as is the case of secant varieties \textcolor{black}{as they are in general not even Cohen-Macaulay (\cite[Theorem 1.3]{CS18}).}

\vspace{5pt}

For this reason, a new definition of varieties having $p$-Du Bois singularities was proposed in \cite{SVV}, which generalizes the previously described notion for varieties with local complete intersection singularities. Our first result is about a vanishing condition on the cohomology of the Du Bois complexes. Following \cite{SVV}, we say that a variety has {\it pre-$p$-Du Bois singularities} if these complexes are concentrated in degree zero for $0\leq k\leq p$ (see \definitionref{defpredb}). When $L$ satisfies $(Q_n)$-property, $\Sigma$ satisfies this for every non-negative integer $p$. In fact, we have:

\begin{intro-theorem}\label{thm1} Let $p\in\mathbb{N}$ and assume $L$ satisfies $(Q_p)$-property. Then the natural maps $$\mathcal{H}^0(\underline{\Omega}^k_{\Sigma})\to \underline{\Omega}^k_{\Sigma}$$ are quasi-isomorphisms for all $0\leq k\leq p$; in particular, the singularities of $\Sigma$ are pre-$p$-Du Bois.
\end{intro-theorem}

We remark that the conclusion of the above holds more generally when $L$ is 3-very ample and satisfies $(Q3_p)$ as our proof shows.

\vspace{5pt}

In addition to requiring $Z$ to have pre-$p$-Du Bois singularities, following \cite{SVV}, a variety is said to have $p$-Du Bois singularities if two extra conditions are satisfied. The first is a codimension condition on the singular locus, and the second is a condition in degree zero for $\underline{\Omega}^{k}_Z$ when $0\leq k\leq p$ (see \definitionref{defDB}). Secant varieties satisfy the first condition up to a certain range; and in relation to the last condition, we show the following where we write $\Omega_{\Sigma}^{[k]}$ for the reflexive hull of $\Omega_{\Sigma}^{k}$:

\begin{intro-theorem}\label{thm2.3} Let $p$ be a positive integer and assume $L$ satisfies $(Q_p)$-property. Then the natural maps $$\delta_k:\cohH^0(\underline{\Omega}^{k}_{\Sigma})\to\Omega_{\Sigma}^{[k]}$$ are isomorphisms for $1\leq k \leq p$ if and only if $H^k(\mathcal{O}_{X})=0$ for $1\leq k\leq p$.
\end{intro-theorem} 

As before, we remark here that if $\Sigma$ is normal, then its singularities are Du Bois and the conclusion of \theoremref{thm2.3} holds if $L$ is 3-very ample and satisfies $(Q3_p)$.

\vspace{5pt}

We now introduce a useful invariant $$\nu(X):=\max\left\{i\mid 1\leq i\leq n-1\,\textrm{ such that }\, H^j(\mathcal{O}_X)=0\,\textrm{ for all }\, 1\leq j\leq i\right\},$$ where we set the convention $\nu(X)=0$ if the set above is empty. The importance of this invariant was observed in \cite{CS18}; it follows from {\it loc. cit.} that if $L$ satisfies $(Q_0)$-property, then $\Sigma$ is Cohen-Macaulay if and only if $\nu(X)=n-1$ (notice that this always holds when $n=1$), and it has rational singularities if and only if $\nu(X)=n-1$ and $H^n(\mathcal{O}_X)=0$.
The following is an immediate consequence of \theoremref{thm1} and \theoremref{thm2.3}. 

\begin{intro-corollary}\label{corpdb} Let $p\in\mathbb{N}$ with $p\leq\lfloor\frac{n}{2}\rfloor$, and assume $L$ satisfies $(Q_p)$-property.
Then $\Sigma$ has $p$-Du Bois singularities if and only if $p\leq \nu(X)$.
\end{intro-corollary}

When $L$ is $3$-very ample, the codimension of the singular locus of $\Sigma$ is $n+1$ unless $\Sigma$ is smooth, whence it follows from the definition of $p$-Du Bois singularity and \propositionref{stronglog}, that in this case $\Sigma$ has $p$-Du Bois singularity for some $p>\lfloor\frac{n}{2}\rfloor$ if and only if $\Sigma=\mathbb{P}(H^0(L))$.

\vspace{5pt}

We further study the morphisms $\delta_k:\mathcal{H}^0(\underline{\Omega}^k_{\Sigma})\to \Omega_{\Sigma}^{[k]}$ and in \theoremref{thm2}, we also discuss the case for higher degrees not considered in the definition of $p$-Du Bois singularities. 
Recall that a variety $Z$ is said to have weakly rational singularities if the Grauert-Riemenschneider sheaf $\omega_Z^{\textrm{GR}}$, which is by definition the push-forward of the canonical bundle from a resolution of singularities of $Z$, is reflexive. When $L$ satisfies $(Q_0)$-property, \cite[Theorem 1.4]{CS18} shows that the natural map $$\omega_{\Sigma}^{\textrm{GR}}\hookrightarrow \omega_{\Sigma}:=\mathcal{H}^{-2n-1}(\omega_{\Sigma}^{\bullet})$$ is an isomorphism if and only if $H^n(\mathcal{O}_X)=0$ (throughout this article, we denote the dualizing complex of a variety $Z$ by $\omega_Z^{\bullet}$; in particular if $Z$ is smooth or even Cohen-Macaulay, we have $\omega_Z^{\bullet}=\omega_Z[\dim Z]$ where $\omega_Z$ is the canonical sheaf). Notice that under our set-up, the map above is an isomorphism if and only if the singularities of $\Sigma$ are weakly rational as $\Sigma$ is normal and $\omega_{\Sigma}$ is reflexive. When $L$ satisfies $(Q_n)$-property, {\it loc. cit.} combined with \theoremref{thm2} (1) establishes the equivalences:
\begin{center}
    $\Sigma$ has weakly rational singularities $\iff H^n(\mathcal{O}_X)=0\iff \delta_{2n+1}$ is an isomorphism\\ $\iff \delta_{2n}$ is an isomorphism.
\end{center}

\vspace{10pt}

\noindent\textbf{Higher rational singularities.} We discuss next the extension of the notion of rational singularities. Recall that a variety $Z$ is said to have rational singularities if for a resolution of singularities $f: \widetilde{Z}\to Z$, the canonical morphism $\mathcal{O}_Z\to {\bf R}f_*\mathcal{O}_{\widetilde{Z}}$ is a quasi-isomorphism. Intrinsically, this is equivalent to requiring that the morphism $\mathcal{O}_Z\to{\bf D}_Z(\underline{\Omega}_Z^{\dim Z})$ is a quasi-isomorphism, where ${\bf D}_Z(-)$ is the Grothendieck dual. Analogous to the definition of pre-$p$-Du Bois singularities, a variety $Z$ is said to have {\it pre-$p$-rational singularities} if the complexes ${\bf D}_Z(\underline{\Omega}_Z^{\dim Z-k})$ are concentrated in degree zero for any $0\leq k\leq p$ (see \definitionref{defprerat}). When $p<\textrm{codim}_Z(Z_{\textrm{sing}})$ and $f:\Tilde{Z}\to Z$ is a strong log resolution with $E:=f^{-1}{(Z_{\textrm{sing}})}_{\textrm{red}}$ a simple normal crossing divisor (see \definitionref{defsl}), this condition is equivalent to requiring that the complexes ${\bf R}f_*\Omega^k_{\widetilde{Z}}(\log E)$ are concentrated in degree zero for $0\leq k\leq p$. It was shown in \cite[Theorem B]{SVV} that for a normal variety, pre-$p$-rational singularities are pre-$p$-Du Bois. However, it turns out that the singularities of secant varieties are almost never pre-$1$-rational: 

\begin{intro-theorem}\label{prerat}
Assume $L$ satisfies $(Q_1)$-property. Then $\Sigma$ has pre-$1$-rational singularities if and only if $X\subset\mathbb{P}(H^0(L))$ is a rational normal curve of degree $\geq 3$.
\end{intro-theorem}

The above result is a consequence of a more general fact that when $L$ satisfies $(Q_1)$-property, we have $\mathcal{H}^1({\bf D}_{\Sigma}(\underline{\Omega}_{\Sigma}^{\dim\Sigma-1}))\neq 0$.

\vspace{5pt}

Following \cite{SVV}, a variety is said to have $p$-rational singularities if $Z$ is normal, has pre-$p$-rational singularities, and additionally satisfies a codimension condition of the singular locus (see \definitionref{defhr}). 
By \propositionref{stronglog}, this last condition is never satisfied when $X\subset\mathbb{P}(H^0(L))$ is a rational normal curve of degree $\geq 4$.

\vspace{5pt}

It is worth mentioning that, unlike the local complete intersection case, where $p$-Du Bois singularities are $(p-1)$-rational (\cite{CDM, MP22, FL}), our study produces secant varieties whose singularities are $p$-Du Bois for some $p\geq 2$ but not even pre-$1$-rational. This feature is also shared by the singularities of non-simplicial affine toric varieties, see \cite[Proposition E]{SVV}. We also obtain secant varieties whose singularities are $p$-Du Bois for large $p$ but not rational, a feature that is shared by the singularities of certain affine cones over smooth projective varieties, see \cite[Proposition F]{SVV}. \\

\noindent\textbf{Some examples.} We provide concrete examples to highlight the scope of our results:

\begin{example}[Curves]
Let $C\subset\mathbb{P}^N$ be a smooth curve of genus $g$, embedded by the complete linear series of a line bundle $L$ with $\textrm{deg}(L)\geq 2g+3$. Then:
\begin{itemize}
    \item $\Sigma$ is normal, Cohen-Macaulay, and has Du Bois singularities. It has weakly rational singularities $\iff$ it has rational singularities $\iff$ $g=0$.
    \item The singularities of $\Sigma$ are pre-$p$-Du Bois for all $p\geq 0$. They are pre-$1$-rational if and only if $g=0$. 
    \item The singularities are not $p$-Du Bois for any $p\geq 1$ unless $\Sigma=\mathbb{P}^N$ (one can show that this happens if and only if $C\subset\mathbb{P}^N$ is a twisted cubic, i.e., $(C,L)\cong (\mathbb{P}^1,\mathcal{O}_{\mathbb{P}^1}(3))$).
\end{itemize}
\end{example}

\begin{example}[Higher dimensions] Let $X\subset\mathbb{P}^N$ be a smooth projective variety of dimension $n\geq 2$, embedded by the complete linear series of $$L:=\left[\binom{n-1}{\lfloor\frac{n-1}{2}\rfloor}+1\right]K_X+\left[\left(\binom{n-1}{\lfloor\frac{n-1}{2}\rfloor}+1\right)(n+2)+2n\left(\binom{n}{\lfloor\frac{n}{2}\rfloor}+1\right)\right]A+B$$ where $A$ and $B$ are very ample and nef line bundles respectively. Then $\Sigma$ is normal and has Du Bois singularities. Moreover, the singularities of $\Sigma$ has the following properties:
\begin{itemize}
    \item They are pre-$p$-Du Bois for all $p\geq 0$. However, they are never pre-$1$-rational.
    \item If $X$ is rationally connected, then the singularities are rational and $\lfloor\frac{n}{2}\rfloor$-Du Bois.
    \item If $X$ is Calabi-Yau in the strong sense (i.e., $H^i(\mathcal{O}_X)=0$ for all $1\leq i\leq n-1$) then they are $\lfloor\frac{n}{2}\rfloor$-Du Bois, but not rational. In this case, $\Sigma$ is Cohen-Macaulay, but not weakly rational. 
    \item If $X$ is hyper-K\"ahler (recall that they live in even dimensions), then they are $1$-Du Bois, but not $2$-Du Bois. In this case, $\Sigma$ is Cohen-Macaulay if and only if $X$ is a K3 surface. However, the singularities of $\Sigma$ are not weakly rational. 
\end{itemize}
\end{example}

One could also obtain many other examples with various features. \\

\noindent\textbf{Consequences.} As the associated graded objects $\underline{\Omega}_Z^k$ of the Du Bois complex of a variety $Z$ are generalizations of the sheaf of $k$-forms in the smooth case, we can use \theoremref{thm1} and \theoremref{thm2.3} to prove a Kodaira-Akizuki-Nakano type vanishing theorem:

\begin{intro-corollary}[Analogue of Kodaira-Akizuki-Nakano vanishing theorem]\label{nakano}
    Let $p$ be a positive integer and assume $L$ satisfies $(Q_p)$-property. Let $\mathcal{L}$ be an ample line bundle on $\Sigma$. If $H^k(\mathcal{O}_{X})=0$ for $1\leq k\leq p$, then $$H^q(\Omega_{\Sigma}^{[p]}\otimes \mathcal{L})=0\,\textrm{ when $p+q>\dim\Sigma=2n+1$}.$$
\end{intro-corollary}

We also apply our results to obtain consequences for $h$-differentials, an introduction to which can be found in \cite{HJ}. It has been proven in \cite{KS} that if $Z$ is a variety with rational singularities, then 
$\Omega^p_h|_{Z}\cong\Omega_{Z}^{[p]}$ for all $p$. Here we obtain the isomorphisms $\Omega^p_h|_{\Sigma}\cong\Omega_{\Sigma}^{[p]}$ up to a certain range of $p$, even when the singularities of $\Sigma$ are not rational:

\begin{intro-corollary}[Description of the $h$-differentials]\label{corhdiff}
    Let $p$ be a positive integer and assume $L$ satisfies $(Q_p)$-property. If $H^k(\mathcal{O}_{X})=0$ for $1\leq k\leq p$, then then there is a natural isomorphism $\Omega^p_h|_{\Sigma}\cong\Omega_{\Sigma}^{[p]}$.
\end{intro-corollary}   

The previous two results are special cases of \corollaryref{nakano'} and \corollaryref{h'}. Finally, pre-$1$-rational singularities of rational normal curves of degree $\geq 3$ give consequences for the Hodge-Du Bois numbers of their secant varieties:

\begin{intro-corollary}[Symmetry of Hodge-Du Bois numbers]\label{hodgesym}
    Let $X\subset \mathbb{P}^{c+1}$ be a rational normal curve of degree $\geq 3$. Then 
    \begin{equation}\label{hdb}
        \underline{h}^{p,q}(\Sigma)=\underline{h}^{q,p}(\Sigma)=\underline{h}^{3-p,3-q}
    \end{equation} for all $0\leq p,q\leq 3$, where $\underline{h}^{p,q}(\Sigma):=\dim\mathbb{H}^q(\Sigma,\underline{\Omega}_{\Sigma}^p)$.
\end{intro-corollary}

As a concluding note, we remark that there is a third measure of singularities that is natural to consider, which is the local cohomological dimension $\textrm{lcd}(\mathbb{P}^N,\Sigma)$ of $\Sigma$ inside $\mathbb{P}^N:=\mathbb{P}(H^0(L))$, and to study the filtrations on the local cohomology sheaves along the direction of \cite{MP}. This will be the topic of a future study.

\vspace{5pt}

The structure of this article can be summarized as follows: Sect. \ref{ch1} is divided in three parts, \S \ref{prelimsing}, \S \ref{prelims} and \S \ref{secpl}. In \S \ref{prelimsing}, we provide a brief review on higher Du Bois and higher rational singularities, and in \S \ref{prelims}, we recall the basics on secant varieties. The objective of \S \ref{secpl} is to study $(Q_p)$-property of pluri-adjoint linear series. The proofs of \theoremref{thm1} and \theoremref{thm2.3} are given in Sect. \ref{sectiondb}. The proof of \theoremref{prerat} appears in Sect. \ref{sectionrat}.  \\

\noindent
{\bf Acknowledgements.} 
We are very grateful to Mircea Musta\c{t}\u{a} and Mihnea Popa for valuable comments on an earlier version of this manuscript and conversations at different stages of this article. We thank Sridhar Venkatesh for several helpful discussions. D.R. expresses his gratitude to Angelo Felice Lopez for valuable comments on an earlier draft. We also thank the anonymous referee for comments and suggestions that improved the exposition. L.S. was partially supported by Guangzhou Science and Technology Programme (No. 2024A04J6409) and NSFC grant (No.~12371063).

\section{Preliminaries}\label{ch1}

We work over the field $\mathbb{C}$ of complex numbers. By a {\it variety}, we mean an integral separated scheme of finite type over $\mathbb{C}$. For a variety $Z$, we denote by $Z_{\textrm{sing}}$ its singular locus. 

\subsection{Hodge theory}\label{prelimsing}
In this section, we recall the basics of higher Du Bois and rational singularities. In pursuit of a general theory of these singularities beyond the local complete intersection case, \cite{SVV} extracted their key features that we will describe.  

\subsubsection{Higher Du Bois singularities}  The first property of this type of singularities is a vanishing condition:
\begin{definition}\label{defpredb}
    A variety $Z$ is said to have {\it pre-$p$-Du Bois singularities} for $p\in\mathbb{N}$ if $$\mathcal{H}^i(\underline{\Omega}_Z^k)=0\,\textrm{ for all }\,i\geq 1, 0\leq k\leq p.$$ Equivalently, the complexes $\underline{\Omega}_Z^k$ are concentrated in degree zero for $k$ in the given range.
\end{definition}

Recall that by definition, a variety $Z$ has Du Bois singularities if the morphism $\mathcal{O}_Z\to\underline{\Omega}_Z^0$ is a quasi-isomorphism.
Also recall that $Z$ is called {\it seminormal} if $\mathcal{H}^0(\underline{\Omega}_Z^0)\cong\mathcal{O}_Z$. In particular, $Z$ has Du Bois singularities if and only if it is seminormal and its singularities are pre-$0$-Du Bois. The picture generalizes through the following
\begin{definition}\label{defDB}
    A variety $Z$ is said to have {\it $p$-Du Bois singularities} for $p\in\mathbb{N}$ if it is seminormal, and the following conditions are satisfied:
    \begin{enumerate}
        \item $\textrm{codim}_Z(Z_{\textrm{sing}})\geq 2p+1$,
        \item $Z$ has pre-$p$-Du Bois singularities,
        \item $\mathcal{H}^0(\underline{\Omega}_Z^k)$ is reflexive for all $0\leq k\leq p$.
    \end{enumerate}
\end{definition}

A related condition on a variety $Z$ is the requirement that the morphisms 
\begin{equation}\label{stDB}
    \Omega_Z^k\to \underline{\Omega}_Z^k\,\textrm{ are quasi-isomorphisms for }\,0\leq k\leq p.
\end{equation} 
The above condition is equivalent to the conditions stated in \definitionref{defDB} when $Z$ is a local complete intersection, but the requirement \eqref{stDB} is generally more restrictive. If $Z$ satisfies \eqref{stDB}, then its singularities are called {\it strict-$p$-Du Bois} in \cite{SVV}.

\subsubsection{Higher rational singularities} We now proceed towards the definition of higher rational singularities. Let us first introduce the notation for the (shifted) Grothendieck duality functor: given a variety $Z$, we set $${\bf D}_Z(-):={\bf R}\mathcal{H}\textit{om}_{\mathcal{O}_Z}(-,\omega_Z^{\bullet})[-\dim Z].$$
While defining higher rational singularities, one is concerned with the complex ${\bf D}_Z(\underline{\Omega}_Z^{\dim Z-k})$:

\begin{definition}\label{defprerat}
    A variety $Z$ is said to have {\it pre-$p$-rational singularities} for $p\in\mathbb{N}$ if $$\mathcal{H}^i({\bf D}_Z(\underline{\Omega}_Z^{\dim Z-k}))=0\,\textrm{ for all }\, i\geq 1,0\leq k\leq p.$$
    Equivalently, ${\bf D}_Z(\underline{\Omega}_Z^{\dim Z-k})$ is concentrated in degree zero for $k$ in the given range.
\end{definition}
Recall that $Z$ is said to have rational singularities if the map $\mathcal{O}_Z\to {\bf D}_Z(\underline{\Omega}_Z^{\dim Z})$ (this is the case $k=0$ of the morphisms $\Omega_Z^k\to {\bf D}_Z(\underline{\Omega}_Z^{\dim Z-k})$ constructed in \cite[Proposition 6.1]{MP22}) is a quasi-isomorphism; this is equivalent to the requirement that $Z$ is normal and $R^if_*\mathcal{O}_{\Tilde{Z}}=0$ when $i>0$ for a log resolution $f:\Tilde{Z}\to Z$ with reduced exceptional divisor $E$ simple normal crossings. In particular, $Z$ has rational singularities if and only if it is normal and its singularities are pre-$0$-rational. In the upcoming sections, we will use the following kind of resolution of singularities.
\begin{definition}\label{defsl}
    Let $Z$ be a variety. By a {\it strong log resolution} of $Z$, we mean a proper morphism $\mu:\Tilde{Z}\to Z$ that is an isomorphism over $Z_\textrm{sm}:=Z\backslash Z_{\textrm{sing}}$ with $\Tilde{Z}$ smooth, and $\mu^{-1}{(Z_{\textrm{sing}})}_{\textrm{red}}$ is a simple normal crossing divisor.
\end{definition}

The picture for rational singularities generalizes through the following: 
\begin{definition}\label{defhr}
    A variety $Z$ is said to have {\it $p$-rational singularities} for $p\in\mathbb{N}$ if it is normal, and 
    \begin{enumerate}
        \item $\textrm{codim}_Z(Z_{\textrm{sing}})> 2p+1$,
        \item $R^if_*\Omega_{\Tilde{Z}}^k(\log E)=0$ for all $i>0$ and $0\leq k\leq p$ and for any strong log resolution $f:\Tilde{Z}\to Z$.
    \end{enumerate}
\end{definition}
The second condition above is equivalent to the requirement that the singularities of $Z$ are pre-$p$-rational. When $Z$ is a local complete intersection, the conditions of \definitionref{defhr} are equivalent to requiring the morphisms 
\begin{equation}\label{stpr}
    \Omega_Z^k\to {\bf D}_Z(\underline{\Omega}_Z^{\dim Z-k})\,\textrm{ are quasi-isomorphisms for }\,0\leq k\leq p.
\end{equation}
However, the above condition is more restrictive in general, and if $Z$ satisfies this, then its singularities are called {\it strict-$p$-rational} in \cite{SVV}. We refer the interested reader to {\it loc. cit.} where various relationships among these singularities are discussed. 

\subsection{Preliminaries on secant varieties}\label{prelims}

In this section, we describe the geometry of secant varieties and provide several computational tools that will be used throughout the rest of the article.

\subsubsection{Strong log resolution of secant varieties}\label{s1}
Let $X$ be a smooth projective variety of dimension $n$. Let $L$ be a very ample line bundle on $X$ inducing the embedding $X\hookrightarrow\mathbb{P}(H^0(L))$. Further, let $X^{[2]}$ be the Hilbert scheme of two points on $X$, which is a smooth projective variety. Consider the universal family $\Phi\subset X^{[2]}\times X$ that comes with two natural projections $q:\Phi\to X$ and $\theta:\Phi\to X^{[2]}$. It is known that $\Phi\cong \textrm{Bl}_{\Delta}(X\times X)$ and let $b_{\Delta}:\Phi\cong\textrm{Bl}_{\Delta}(X\times X)\to X\times X$ be the blow-up morphism where $\Delta\subset X\times X$ is the diagonal. We have the following commutative diagram:
\begin{equation*}
    \begin{tikzcd}
    & \Phi\arrow[dl, swap, "\theta"]\arrow[rr, "b_{\Delta}"]\arrow[dr, "q"] & & X\times X\arrow[dl, swap, "p_1"]\\
    X^{[2]} && X
\end{tikzcd}
\end{equation*}
We set $\mathcal{E}_L:=\theta_*q^*L$ and notice that this bundle is globally generated as $L$ is very ample, i.e., the evaluation map
\begin{equation*}\label{ev}
    H^0(\mathcal{E}_L)\otimes\mathcal{O}_{X^{[2]}}\to \mathcal{E}_L
\end{equation*}
is surjective. Observe that $H^0(\mathcal{E}_L)\cong H^0(L)$ whence the above surjection induces a map $f:\mathbb{P}(\mathcal{E}_L)\to\mathbb{P}(H^0(L))$ which surjects onto the secant variety $\Sigma:=\Sigma(X,L)$ which, by definition, is the Zariski closure of the union of 2-secant lines of $X\hookrightarrow\mathbb{P}(H^0(L))$. This way, we obtain the surjective map $t:\mathbb{P}(\mathcal{E}_L)\to \Sigma$. In what follows, we set $\pi:\mathbb{P}(\mathcal{E}_L)\to X^{[2]}$ to be the structure morphism. Recall the following
\begin{definition}\label{defva}
A bundle $L$ on a smooth projective variety $X$ is called {\it $k$--very ample} if for any $0$--dimensional subscheme $\xi$ of length $k+1$, the evaluation map of global sections
$H^0(L)\to H^0(L\otimes \mathcal{O}_{\xi})$
surjects. 
\end{definition}

The singular locus of secant varieties is very simple under the mild condition that $L$ is 3-very ample. Moreover, in this case the map $t$ described above is a resolution of singularities. Although these facts are well-known to experts, 
we include the proofs for the convenience of the reader. 

\begin{proposition}\label{stronglog}
Assume $L$ is 3-very ample. Then the following statements hold:
\begin{enumerate}
    \item\label{stronglog1} $t|_{\mathbb{P}(\mathcal{E}_L)\backslash t^{-1}(X)}: \mathbb{P}(\mathcal{E}_L)\backslash t^{-1}(X)\to \Sigma\backslash X$ is an isomorphism. In particular $\Sigma_{\textrm{sing}}\subseteq X$.
    \item\label{stronglog2} If in addition $\Sigma\neq\mathbb{P}(H^0(L))$, then $\Sigma_{\textrm{sing}}=X$.
\end{enumerate} 
\end{proposition}

We proceed to the proof of the above result. In what follows, we use that given a zero-dimensional subspace $\xi\subset X$ with ideal sheaf $\mathcal{I}_{\xi}$, the linear subspace $\langle\xi\rangle$ spanned by $\xi$ is isomorphic to $\mathbb{P}(H^0(L)/H^0(L\otimes\mathcal{I}_{\xi}))$. Observe that if $\xi_1\subset X$ and $\xi_2\subset X$ are distinct zero-dimensional subschemes of length 2, then $\langle\xi_1\rangle\neq\langle\xi_2\rangle$ when $L$ is $3$-very ample.

\begin{lemma}\label{c1}
Assume $L$ is $3$-very ample. Let $\xi_1\subset X$ and $\xi_2\subset X$ be zero-dimensional subschemes of length 2 with $\textrm{Supp}(\xi_1)\cap\textrm{Supp}(\xi_2)=\emptyset$. Then $\langle\xi_1\rangle\cap\langle\xi_2\rangle=\emptyset$.
\end{lemma}
\begin{proof} Since $\xi:=\xi_1\cup\xi_2$ has length 4 and $L$ is 3-very ample, the natural map $H^0(L)\to H^0(L\otimes\mathcal{O}_{\xi})$ is surjective. It follows that $\langle\xi\rangle\cong\mathbb{P}(H^0(L\otimes\mathcal{O}_{\xi}))\cong\mathbb{P}^3$. Therefore the lines $\langle\xi_1\rangle$ and $\langle\xi_2\rangle$ do not intersect. \end{proof}

\begin{lemma}\label{c2}
Assume $L$ is $3$-very ample, and let $x\in\Sigma$. If there is a unique zero-dimensional subscheme $\xi\subset X$ of length 2 such that $x\in\langle\xi\rangle$, then there exists an open set $x\in V\subset \Sigma$ such that $t|_{t^{-1}(V)}:t^{-1}(V)\to V$ is an isomorphism; in particular $x\in \Sigma$ is smooth.
\end{lemma}
\begin{proof} Put $l=\langle\xi\rangle$ and consider the Cartesian square:
\[
\begin{tikzcd}
    t^{-1}(x)\arrow[d]\arrow[r, hook] & \pi^{-1}(\xi)\arrow[d, "\cong"]\\
    \left\{x\right\}\arrow[r,hook] & l
\end{tikzcd}
\]
Since $t|_{\pi^{-1}(\xi)}:\pi^{-1}([\xi])\to l$ is an isomorphism by the construction of the resolution $t$, the scheme-theoretic preimage $t^{-1}(x)$ is a reduced point. Thus, there exists an open set $x\in U\subset \Sigma$ such that for all $y\in U$, we have $\dim(t^{-1}(y))=0$. Consequently, $t|_U:t^{-1}(U)\to U$, is a finite morphism. Using base change of affine morphisms, and upper semi-continuity of ranks of coherent sheaves, we conclude that there exists an open subset $x\in V\subseteq U$ such that $t|_V:t^{-1}(V)\to V$ is an isomorphism. The last assertion follows since $t^{-1}(V)$ is smooth.
\end{proof}

\begin{proof}[Proof of \propositionref{stronglog} (\ref{stronglog1}).] Thanks to  \lemmaref{c2}, it is enough to show that for any $x\in \Sigma\backslash X$, there exists a unique zero-dimensional subscheme $\xi\subset X$ of length 2 such that $x\subset \langle\xi\rangle$. Suppose to the contrary, $x\in \Sigma\backslash X$ with 
$x\in \left(l_1=\langle\xi_1\rangle\right)\cap \left(l_2=\langle\xi_2\rangle\right)$, where $\xi_i\subset X$ are distinct zero-dimensional subschemes of length 2. Since $\langle\xi_1\rangle\neq\langle\xi_2\rangle$ and since two lines intersect at most at one point, we conclude that $\textrm{Supp}(\xi_1)\cap\textrm{Supp}(\xi_2)=\emptyset$ (see the left schematic diagram in \figureref{fig1}). 
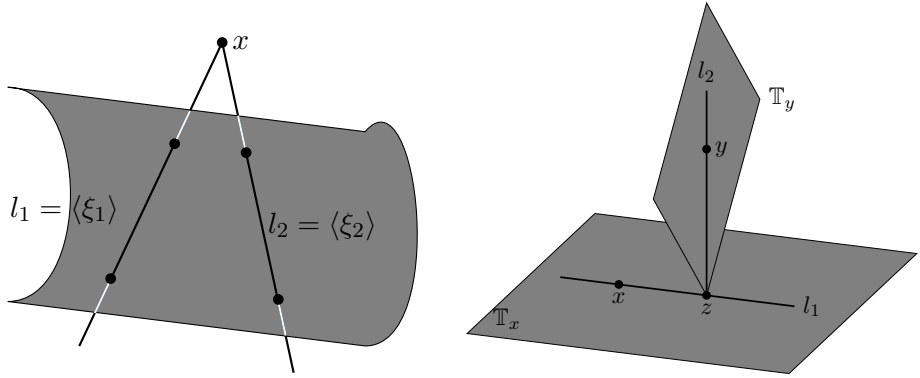
\begin{figure}[H]
    \centering
    \tdplotsetmaincoords{70}{110}
    \resizebox{0.4\textwidth}{!}{
\begin{tikzpicture}[tdplot_main_coords,font=\sffamily]
\draw[fill=gray,opacity=0.2] (0,-3,2) -- (0,2,2) to[out=50,in=0] (0,2,-1) -- (0,-3,-1) to[out=0,in=0] cycle;
\draw[thick, postaction={-,draw=aliceblue,dash pattern=on 0pt off 0.48cm on 0.48cm off 1.50cm}] (0,-2,-1.5) -- (0,0,3);
\draw[thick, postaction={-,draw=aliceblue,dash pattern=on 0pt off 0.50cm on 0.48cm off 1.50cm}] (0,1,-1.5) -- (0,0,3);
\fill (0,-1.56,-0.5) circle[radius=2pt];
\fill (0,-0.665,1.5) circle[radius=2pt];
\fill (0,0.335,1.5) circle[radius=2pt];
\fill (0,0.785,-0.5) circle[radius=2pt];
\fill (0,0,3) circle[radius=2pt] node[right] {$x$};
\fill (0,-1.3,0.5) node[left] {$l_1=\langle\xi_1\rangle$};
\fill (0,0.5,0.5) node[right] {$l_2=\langle\xi_2\rangle$};
\end{tikzpicture}}
\resizebox{0.4\textwidth}{!}{
\begin{tikzpicture}[tdplot_main_coords,font=\sffamily]
\draw[fill=gray,opacity=0.1] (-3,-3,0) -- (-3,2.5,0) -- (3,2.5,0) -- (3,-3,0) -- cycle;
\draw[fill=gray,opacity=0.2] (0,0,0) -- (2.5,0,2.5) -- (0,0,5) -- (-2.5,0,2.5) -- cycle;
\draw[thick] (2.0,-2.3,0) node[left] {\textcolor{black}{$\mathbb{T}_{x}$}};
\draw[thick] (-2.5,0,2.5) node[right] {\textcolor{black}{$\mathbb{T}_{y}$}};
\draw[thick] (0,1.5,0) node[right] {\textcolor{black}{$l_1$}};
\draw[thick] (0,0,3.5) node[above] {\textcolor{black}{$l_2$}};
\draw[thick] (0,-2.5,0) -- (0,1.5,0);
\draw[thick] (0,0,0) -- (0,0,3.5);
\fill (0,-1.5,0) circle[radius=2pt] node[below] {$x$};
\fill (0,0,2.5) circle[radius=2pt] node[right] {$y$};
\fill (0,0,0) circle[radius=2pt] node[below] {$z$};
\end{tikzpicture}}
\caption{Left: $x\in \left(l_1=\langle\xi_1\rangle\right)\cap \left(l_2=\langle\xi_2\rangle\right)$. Right: $z\in \mathbb{T}_x\cap\mathbb{T}_y$ where $z\notin\left\{x,y\right\}$}
\label{fig1}
\end{figure}
\vspace{-0.6cm}
\noindent But this contradicts \lemmaref{c1}.
\end{proof}

We now proceed to prove \propositionref{stronglog} (\ref{stronglog2}).

\begin{lemma}\label{c3}
Assume $L$ is $3$-very ample. Let $x,y\in X$ be two distinct points with the embedded tangent spaces $\mathbb{T}_x\cap\mathbb{T}_y\neq\emptyset$. Then either $\mathbb{T}_x\cap\mathbb{T}_y=\left\{x\right\}$ or $\mathbb{T}_x\cap\mathbb{T}_y=\left\{y\right\}$.
\end{lemma}

\begin{proof} For the sake of contradiction, assume $z\in \mathbb{T}_x\cap\mathbb{T}_y$ where $z\notin\left\{x,y\right\}$. Consider the lines $l_1=\overline{xz}$, $l_2=\overline{yz}$ (see the right schematic diagram in \figureref{fig1}). Notice that $l_1$ (resp. $l_2$) intersects $X$ at $x$ (resp. $y$) with multiplicity $\geq 2$. Consequently, we get zero-dimensional subschemes $\xi_1\subset X$ and $\xi_2\subset X$ of length 2, with $\textrm{Supp}(\xi_1)=\left\{x\right\}$ and $\textrm{Supp}(\xi_2)=\left\{y\right\}$ with $z\in \langle\xi_1\rangle\cap\langle\xi_2\rangle$. This contradicts \lemmaref{c1}.
\end{proof}

\begin{lemma}\label{c4}
Assume $L$ is $3$-very ample. Let $x\in X$. Then for general $y\in X$, $\mathbb{T}_x\cap\mathbb{T}_y=\emptyset$.
\end{lemma}

\begin{proof} Since $X\subset \mathbb{P}(H^0(L))$ is non-degenerate (in particular, $X\not\subseteq\mathbb{T}_x$), for general $y\in X$, we have $\mathbb{T}_x\cap\mathbb{T}_y\neq\left\{y\right\}$. For $y_1,y_2\in X$ distinct points, both distinct from $x$, assume $\mathbb{T}_x\cap\mathbb{T}_{y_1}=\mathbb{T}_x\cap\mathbb{T}_{y_2}=\left\{x\right\}$. Then, as in the proof of \lemmaref{c3}, working with $\overline{xy_1}$ and $\overline{xy_2}$, we obtain length 2 subschemes $\xi_1,\xi_2\subset X$ with $\textrm{Supp}(\xi_1)\cap\textrm{Supp}(\xi_2)=\emptyset$ and $x\in \langle\xi_1\rangle\cap\langle\xi_2\rangle$, which contradicts \lemmaref{c1}. The conclusion follows from \lemmaref{c3}.
\end{proof} 

\begin{proof}[Proof of \propositionref{stronglog} (\ref{stronglog2}).] In view of (1), we suppose to the contrary that $\Sigma_{\text{sing}} \subsetneq X$, and fix $x\in X\backslash{\Sigma_{\text{sing}}}$. Then $\dim \mathbb{T}_x\Sigma=2n+1$. For general $y\in X$, applying Terracini's lemma (cf.~\cite[Lemma 3.4.28]{Laz}), we have that $\textrm{span}(\mathbb{T}_x,\mathbb{T}_y)\subseteq \mathbb{T}_x\Sigma$. By \lemmaref{c4}, $\dim(\textrm{span}(\mathbb{T}_x,\mathbb{T}_y))=2n+1=\dim \mathbb{T}_x\Sigma$. Thus $\textrm{span}(\mathbb{T}_x,\mathbb{T}_y)=\mathbb{T}_x\Sigma$. On the other hand, by the generality of $y$, $\Sigma$ is smooth at $y$, so another application of Terracini's lemma yields that $\textrm{span}(\mathbb{T}_x,\mathbb{T}_y)=\mathbb{T}_y\Sigma$. Therefore $X\subseteq \mathbb{T}_x\Sigma$. Recall that $X\subset\mathbb{P}(H^0(L))$ is non-degenerate, so  $\mathbb{T}_x\Sigma=\mathbb{P}(H^0(L))$. Since $x$ is a smooth point of $\Sigma$, we obtain $\Sigma=\mathbb{P}(H^0(L))$, a contradiction.
\end{proof}

From here until the end of \S \ref{s2}, we tacitly assume that $L$ is 3-very ample. In this case, $\Sigma_{\textrm{sing}}\subseteq X$ and $\mathbb{P}(\mathcal{E}_L)\subset X^{[2]}\times\mathbb{P}(H^0(L))$ together with the second projection provides a natural resolution of singularities $t:\mathbb{P}(\mathcal{E}_L)\to \Sigma(X,L)$ by \propositionref{stronglog}. It follows from \cite[Lemma 3.8]{Ver} that scheme-theoretically we can identify the exceptional divisor $t^{-1}(X)\cong\Phi$ and the restriction of $t$ on it coincides with the surjection $q:\Phi\to X$. \\

As an immediate consequence of the above discussion and \propositionref{stronglog}, we obtain

\begin{corollary}
    Assume $L$ is $3$-very ample and $\Sigma\neq\mathbb{P}(H^0(L))$. Then the morphism $t$ is a strong log resolution of $\Sigma$.
\end{corollary}

Strictly speaking, we don't use that $t$ is a {\it strong} log resolution when $\Sigma\neq\mathbb{P}(H^0(L))$ in the proof of \theoremref{prerat}. This is because, to check whether $\Sigma$ has pre-$1$-rational singularities through its birational description, we only need a resolution which is an isomorphism outside a locus of codimension at least two. The morphism $t$ satisfies this when $L$ is $3$-very ample as in this case $\textrm{codim}_{\Sigma}(X)=n+1\geq 2$. See \remarkref{alex} for more details.

\medskip

In summary, for any $x\in X$ we have the following diagram with Cartesian squares where the vertical arrows are surjections
\begin{equation}\label{ulldiag}
    \begin{tikzcd}
    F_x\arrow[r, hook]\arrow[d] & \Phi\arrow[r, hook]\arrow[d, "q"] & \mathbb{P}(\mathcal{E}_L)\arrow[d, "t"]\arrow[dr, "f"] &\\
    \{x\}\arrow[r, hook] & X\arrow[r, hook] & \Sigma\arrow[r, hook] & \mathbb{P}(H^0(L))
\end{tikzcd}
\end{equation}
and $F_x\cong\textrm{Bl}_xX$, the blow-up of $X$ at $x$. We set $b_x:F_x\cong\textrm{Bl}_xX\to X$ to be the blow-up morphism. In the sequel, we will often use the fact that the map $q:\Phi\to X$ is smooth by \cite[Lemma 2.1]{CS18} without any further reference. 

\subsubsection{Useful isomorphisms and exact sequences}\label{s2} 
We start by recalling some basic facts about the log resolution of $\Sigma$ described in \S \ref{s1} that are used crucially in the proofs of our main results. 
\smallskip

First of all, by \cite[Proof of Lemma 2.3]{Ull16}, we have the isomorphisms
\begin{equation}\label{0}
    \mathcal{N}_{F_x/\Phi}^*\cong\mathcal{O}_{F_x}^{\oplus n}\textrm{, and } \mathcal{N}^*_{\Phi/\mathbb{P}(\mathcal{E}_L)}|_{F_x}\cong b_x^*L(-2E_x),
\end{equation}
where $E_x$ is the exceptional divisor of $b_x$. Moreover, by \cite[Proof of Lemma 2.3]{Ull16}, the normal bundle sequence of $F_x\subset\Phi\subset\mathbb{P}(\mathcal{E}_L)$ is split. Consequently, by using \eqref{0} one obtains the following isomorphism: 
\begin{equation}\label{1}
    \mathcal{N}_{F_x/\mathbb{P}(\mathcal{E}_L)}^*\cong \mathcal{O}_{F_x}^{\oplus n}\oplus b_x^*L(-2E_x).
\end{equation}
Denoting the ideal sheaf of $F_x\subset\mathbb{P}(\mathcal{E}_L)$ by $\mathcal{I}_{F_x}$, and using \eqref{1}, we also obtain the isomorphisms
\begin{equation}\label{sym}
    \mathcal{I}_{F_x}^j/\mathcal{I}_{F_x}^{j+1}\cong\textrm{Sym}^j\mathcal{N}_{F_x/\mathbb{P}(\mathcal{E}_L)}^*\cong \bigoplus\limits_{m=0}^j\left[b_x^*(mL)(-2mE_x)\right]^{\oplus \binom{n+j-m-1}{n-1}}.
\end{equation}

Next, observe that we have the following commutative diagram with exact rows and columns
for $p\geq 1$:
\begin{equation}\label{omega}
\begin{tikzcd}
    & & 0\arrow[d] & 0\arrow[d] &\\
    0\arrow[r] & \Omega^p_{\mathbb{P}(\mathcal{E}_L)}(\log\Phi)(-\Phi)\arrow[r]\arrow[d, equal] & \Omega^p_{\mathbb{P}(\mathcal{E}_L)}\arrow[r]\arrow[d] & \Omega^p_{\Phi}\arrow[r]\arrow[d] & 0\\
    0\arrow[r] & \Omega^p_{\mathbb{P}(\mathcal{E}_L)}(\log\Phi)(-\Phi)\arrow[r] & \Omega^p_{\mathbb{P}(\mathcal{E}_L)}(\log\Phi)\arrow[r]\arrow[d] & \Omega^p_{\mathbb{P}(\mathcal{E}_L)}(\log\Phi)|_{\Phi}\arrow[r]\arrow[d] & 0\\
    & & \Omega^{p-1}_{\Phi}\arrow[r,equal]\arrow[d] &\mathcal{F}\arrow[d] &\\
    && 0 & 0 &
\end{tikzcd}
\end{equation}
Since $\Omega_{\Phi}^{p-1}$ is locally free, for any $x\in X$, restricting the resulting exact sequence appearing in the right vertical column on $F_x$, we obtain the following short exact sequence for any $p\geq 1$:
\begin{equation}\label{ex1}
    0\to{\Omega}^p_{\Phi}|_{F_x}\to \Omega^p_{\mathbb{P}(\mathcal{E}_L)}(\log\Phi)|_{F_x}\to\Omega^{p-1}_{\Phi}|_{F_x}\to 0.
\end{equation}
The above exact sequence will be essential for us in the sequel.\\

We now prove a proposition crucially needed in the proof of \theoremref{thm2.3}: 
  
\begin{proposition}\label{iffeq} There is an isomorphism
\begin{equation}\label{iffeq1}
    H^k(\Phi, \mathcal{O}_{\Phi}) \cong \bigoplus\limits_{j=0}^k H^{k-j}(X, H^j(X,\mathcal{O}_X)\otimes \mathcal{O}_X).
\end{equation}
In particular, 
\begin{equation}\label{iffeq2}
    h^{0,k}(\Phi) = h^{0,k}(X)h^{0,0}(X) + h^{0,k-1}(X)h^{0,1}(X) + \cdots + h^{0,0}(X)h^{0,k}(X).
\end{equation}
\end{proposition}

\begin{proof} We note first that since the map $q$ is smooth, we have an isomorphism 
\begin{equation}\label{del}
    \textbf{R}q_*\mathcal{O}_{\Phi} \cong \bigoplus R^jq_*\mathcal{O}_{\Phi}[-j]
\end{equation} in ${\bf D}^b(\text{Coh}(X))$, the bounded derived category of coherent sheaves on $X$ \cite[Theorem 6.1]{deligne 68} (note that this is an instance of the Decomposition Theorem and taking the Hodge degree 0 \cite[Theorem 1]{saito88} in the simpler case when the map is smooth and projective).
Moreover, 
\begin{equation}\label{cs}
    R^jq_*\mathcal{O}_{\Phi} \cong H^j(X,\mathcal{O}_X)\otimes \mathcal{O}_X
\end{equation} 
by \cite[Lemma 2.2]{CS18}. Taking hypercohomology $\mathbb{H}^k$ of \eqref{del}, this says that $$H^k(\Phi, \mathcal{O}_{\Phi}) \cong \bigoplus H^{k-j}(X, H^j(X,\mathcal{O}_X)\otimes \mathcal{O}_X)$$ which is \eqref{iffeq1}. Lastly, \eqref{iffeq2} follows immediately from this.
\end{proof}

We remark that there is a more elementary proof of \eqref{iffeq2} using the fact that $\Phi$ is the blow-up of $X\times X$ along the diagonal. Lastly, we compute the direct and higher direct images of $\Omega_{\Phi}^1$ that will be required in the proof of \theoremref{prerat}:

\begin{lemma}\label{ext}
    The following statements hold:
    \begin{enumerate}
        \item If $n=1$, then we have the following isomorphisms for all $j$:
        $$R^jq_*\Omega_{\Phi}^1\cong \left[H^j(\mathcal{O}_X)\otimes \Omega_X^1\right]\oplus \left[H^j(\Omega_X^1)\otimes\mathcal{O}_X\right].$$
        \item Assume $n\geq 2$. Then:
        \begin{itemize}
            \item[(i)] $R^jq_*\Omega_{\Phi}^1\cong \left[H^j(\mathcal{O}_X)\otimes \Omega_X^1\right]\oplus \left[H^j(\Omega_X^1)\otimes\mathcal{O}_X\right]$ for all $j\neq 1$;
            \item[(ii)] We have an exact sequence $$0\to \left[H^1(\mathcal{O}_X)\otimes \Omega_X^1\right]\oplus \left[H^1(\Omega_X^1)\otimes \mathcal{O}_X\right]\to R^1q_*\Omega_{\Phi}^1\to\mathcal{O}_X\to 0.$$
        \end{itemize}
    \end{enumerate}
\end{lemma}

\begin{proof} Let us denote by $p_1$ and $p_2$ the two projections from $X\times X$ to its factors. Using the diagram 
\[
\begin{tikzcd}
    X\times X\arrow[r, "p_2"]\arrow[d, "p_1"] & X\arrow[d, "q_2"]\\
    X\arrow[r, "q_1"] & \bullet
\end{tikzcd}
\]
and flat base change, we deduce that $R^j{p_1}_*p_2^*\Omega_X^1\cong q_1^*R^j{q_2}_*\Omega_X^1\cong H^j(\Omega_X^1)\otimes \mathcal{O}_X$. Combining this with projection formula and $R^j{p_1}_*\mathcal{O}_{X\times X}\cong H^j(\mathcal{O}_X)\otimes \mathcal{O}_X$, we obtain for all $j$
\begin{equation}\label{rjexplicit}
    R^j{p_1}_*\Omega_{X\times X}^1\cong R^j{p_1}_*(p_1^*\Omega_X^1\oplus  p_2^*\Omega_X^1)\cong \left[H^j(\mathcal{O}_X)\otimes \Omega_X^1\right]\oplus \left[H^j(\Omega_X^1)\otimes\mathcal{O}_X\right].
\end{equation}
We recall that $q=p_1\circ b_{\Delta}$ and note that we have the following commutative diagram:
\[
\begin{tikzcd}
    \mathbb{P}^{n-1}\cong E_x\arrow[r, hook]\arrow[d] & E_{\Delta}\arrow[r, hook, "j_{\Delta}'"]\arrow[d, "q_{\Delta}"] & \Phi\cong\textrm{Bl}_{\Delta}(X\times X)\arrow[d, "b_{\Delta}"]\arrow[dd, bend left=60, "q"]\\
    \left\{(x,x)\right\}\arrow[r, hook]& \Delta\arrow[r, hook, "j_{\Delta}"]\arrow{dr}{\cong}[swap]{p_0} & X\times X\arrow[d, "p_1"]\\
    & & X
\end{tikzcd}
\]
When $n=1$, $b_{\Delta}$ is an isomorphism, whence $R^jq_*\Omega_{\Phi}^1\cong R^j{p_1}_*\Omega_{X\times X}^1$ and the conclusion follows from \eqref{rjexplicit}. This proves (1).

Now assume $n\geq 2$. Recall that $R^j{b_{\Delta}}_*\mathcal{O}_{\Phi}=0$ for all $j\geq 1$, and ${b_{\Delta}}_*\mathcal{O}_{\Phi}\cong \mathcal{O}_{X\times X}$. Using Leray spectral sequence and projection formula, we obtain the following isomorphisms for all $j$: 
\begin{equation}\label{isolcd}
    R^jq_*b_{\Delta}^*\Omega_{X\times X}^1\cong R^j{p_1}_*\Omega_{X\times X}^1\,\textrm{ and } R^jq_*{j'_{\Delta}}_*\Omega_{E_{\Delta}/\Delta}^1\cong {p_0}_*R^j{q_{\Delta}}_*\Omega_{E_{\Delta}/\Delta}^1.
\end{equation}
Notice that $q_{\Delta}:E_{\Delta}\cong \mathbb{P}({\mathcal{N}^*_{\Delta}})\to\Delta$ is the structure morphism of the projective bundle, where $\mathcal{N}_{\Delta}$ is the normal bundle of $\Delta\hookrightarrow X\times X$. Consequently, passing to the long exact sequence corresponding to ${q_\Delta}_*$ of the exact sequence $$0\to \Omega_{E_{\Delta}/\Delta}^1\to q_{\Delta}^*\mathcal{N}^*_{\Delta}(-1)\to\mathcal{O}_{E_{\Delta}}\to 0,$$ we obtain 
\begin{equation}\label{isolcd'}
    R^j{q_{\Delta}}_*\Omega_{E_{\Delta}/\Delta}^1\cong \begin{cases} 
      \mathcal{O}_{\Delta} & \textrm{if }\, j=1; \\
      0 & \textrm{otherwise}.
   \end{cases}
\end{equation}
Also, passing to the long exact sequence corresponding to $q_*$ of the following short exact sequence 
$$0\to b_{\Delta}^*\Omega_{X\times X}^1\to\Omega_{\Phi}^1\to {j'_{\Delta}}_*\Omega_{E_{\Delta}/\Delta}^1\to 0,$$
and using \eqref{isolcd}, \eqref{isolcd'} we obtain the isomorphisms
\begin{equation}\label{isolcd''}
\begin{gathered}
    q_*\Omega_{\Phi}^1\cong {p_1}_*\Omega_{X\times X}^1,\\
    R^jq_*\Omega_{\Phi}^1\cong R^j{p_1}_*\Omega_{X\times X}^1\,\textrm{ for all }\, j\geq 3.
\end{gathered}
\end{equation}
Moreover, we also obtain the following exact sequence
\begin{equation}\label{leslcd}
    0\to R^1{p_1}_*\Omega_{X\times X}^1\to R^1q_*\Omega_{\Phi}^1\to \mathcal{O}_X\to R^2{p_1}_*\Omega_{X\times X}^1\to R^2q_*\Omega_{\Phi}^1\to 0.
\end{equation}
Now, the map $$\mathcal{O}_X\to R^2{p_1}_*\Omega_{X\times X}^1$$ is injective if it is non-zero. We claim that it is the zero map. Indeed, for otherwise $R^1{p_1}_*\Omega_{X\times X}^1\cong R^1q_*\Omega_{\Phi}^1$. The Leray spectral sequence $$E_2^{i,j}:=H^i(R^jq_*\Omega_{\Phi}^1)\implies H^{i+j}(\Omega_{\Phi}^1),$$
being a first quadrant spectral sequence, induces the exact sequence $$0\to E_2^{1,0}\to H^1(\Omega_{\Phi}^1)\to E_2^{0,1}\to E_2^{2,0}.$$ We conclude that 
\begin{equation}\label{ineq1}
    h^0(R^1{p_1}_*\Omega_{X\times X}^1)=h^0(R^1q_*\Omega_{\Phi}^1)\geq h^1(\Omega_{\Phi}^1)-h^1(q_*\Omega_{\Phi}^1).
\end{equation}
By \eqref{rjexplicit}, we compute 
\begin{equation}\label{ineq2}
    h^0(R^1{p_1}_*\Omega_{X\times X}^1)=(h^{1,0}(X))^2+h^{1,1}(X).
\end{equation}
On the other hand, $h^1(\Omega_{\Phi}^1)-h^1(q_*\Omega_{\Phi}^1)=h^1(\Omega_{X\times X}^1)+1-h^1(q_*\Omega_{\Phi}^1)$ as $\Phi\cong \textrm{Bl}_{\Delta}(X\times X)$, whence using \eqref{rjexplicit} once again, this time to calculate $h^1(q_*\Omega_{\Phi}^1)$ using \eqref{isolcd''}, we obtain 
\begin{equation}\label{ineq3}
    h^1(\Omega_{\Phi}^1)-h^1(q_*\Omega_{\Phi}^1)=(h^{1,0}(X))^2+h^{1,1}(X)+1.
\end{equation}
But \eqref{ineq2} and \eqref{ineq3} contradicts \eqref{ineq1}. Thus, \eqref{leslcd} breaks off into the desired short exact sequence and the isomorphism. This, combined with \eqref{isolcd''} and \eqref{rjexplicit} proves (2). 
\end{proof}

\subsection{$(Q_p)$-property and pluri-adjoint series}\label{secpl} Let $X$ be a smooth projective variety of dimension $n$. Let $\mathcal{I}_x$ be the ideal sheaf of $x\in X$. We now formally introduce the following:

\begin{definition}\label{qp}
Let $p\geq 0$ be an integer, and let $L$ be a 3-very ample line bundle on $X$. Then $L$ is said to satisfy {\it $(Q_p)$-property} if the following conditions are satisfied for all $x\in X$:
\begin{description}
    \item[$(Q1)$] the natural map $\text{Sym}^iH^0(L\otimes \mathcal{I}_x^2)\to H^0(L^{\otimes i}\otimes \mathcal{I}_x^{2i})$ is surjective for all $i\geq 1$ (this is equivalent to requiring $b_x^*L(-2E_x)$ is projectively normal),
    \item[$(Q2)$] $b_x^*L(-2E_x)$ is ample\footnote{If we assume $L$ is 3-jet ample (c.f. page 16), then $(Q2)$ is automatic by \cite[Proposition 2.7]{BDS99}. Also, an interested reader may verify that if an arbitrary line bundle $L$ (not necessarily assumed to be 3-very ample) satisfies $(Q1)$ and $(Q2)$, then it is almost immediate from an observation of Mumford (\cite[page 38]{Mum}) that $b_x^*L(-2E_x)$ is very ample for all $x\in X$, whence $L$ is also very ample.},
    \item[$(Q3_p)$] $H^i(\Omega^q_{F_x}\otimes b_x^*(jL)(-2jE_x))=0$ for all $i,j\geq 1$, $0\leq q\leq p$.
\end{description}
\end{definition}

Examples where $(Q3_p)$ is satisfied include the case when $(Q2)$ is satisfied and $F_x$ satisfies Bott vanishing for all $x\in X$. We proceed to show that the line bundles as in the set-up below satisfy $(Q_p)$-property for suitable $p$.

\begin{set-up}\label{setup}Let $X$ be a smooth projective variety of dimension $n$. Let $L$ be a line bundle on $X$ that satisfies:
\begin{itemize}
    \item If $n=1$, we assume $\mathrm{deg}(L)\geq 2g+3$.
    \item If $n\geq 2$, then $L=L_{l,d}:=lK_X+dA+B$ where $A$ and $B$ are very ample and nef line bundles respectively. We additionally assume for a given $s\in\mathbb{N}$ (which will be specified in the statement of our results) that $(l,d)\in\mathbb{N}\times\mathbb{N}$ satisfies the following conditions\footnote{Here we introduce the convention that $\binom{a}{b}=0$ if $b<0$ or if $b>a$ or if $a=0$.}:
\begin{equation}\label{co3}
\begin{array}{c}
l\geq \max\limits_{-1\leq i\leq s-1}\left\{\binom{n-1}{i}+1\right\},\\[10pt]
d\geq \max\left\{\begin{array}{l}
  l(n+2)+2,l(n+1)+2+s,(n+1)(l+1), \max\limits_{-1\leq i\leq s-1}\left\{d_{n,i}\right\}   
\end{array}\right\}
\end{array}
\end{equation}
where 
\small
$$d_{n,i}=\max\left\{(n+2)\left(l-\binom{n-1}{i}-1\right)+2(i+1)\left(\binom{n}{i+1}+1\right), 2(i+1)\left(\binom{n}{i+1}+1\right)+1\right\}.$$
\end{itemize}  
\end{set-up}

\normalsize
Two remarks are in order:

\begin{remark}\label{rmkva}
    We note that in the situation of \setupref{setup} (by which we mean, in particular, that $(l,d)$ satisfies \eqref{co3} for some $s\geq 0$ when $n\geq 2$), $L$ is 3-very ample. This is evident for $n=1$. To see this for $n\geq 2$, use the well-known fact that $K_X+(n+2)A+B'$ is very ample for any nef line bundle $B'$ as $A$ is very ample. Write $$L_{l,d}=[K_X+(n+2)A+B]+[(l-1)K_X+(d-n-2)A].$$ Recall that a line bundle is $1$-very ample is equivalent to it being very ample, and by \cite[Theorem 1.1]{HTT}, a tensor product of $a$ and $b$-very ample line bundles is $(a+b)$-very ample when $a,b\geq 0$. To this end, note that $d-n-2\geq (l-1)(n+2)+2$ as $d\geq l(n+2)+2$ by assumption, whence $$L_{l,d}=\underbrace{[K_X+(n+2)A+B]}_{\text{very ample}}+(l-1)\underbrace{[K_X+(n+2)A]}_{\text{very ample}}+jA \textrm{ with $j\geq 2$}$$ which proves the assertion. In fact, let us also note that $L_{l,d}$ is 3-very ample if $l=2$ (resp. $l\geq 3$) and $d\geq l(n+2)+1$ (resp. $d\geq l(n+2)$).
\end{remark}

\begin{remark}\label{rmkndb}
    We note that in the situation of \setupref{setup}, $L$ satisfies $(Q1)$ and $(Q3_0)$. For $n=1$, this follows from \cite[Proof of Corollary A]{Ull16} and \cite[Proof of Theorem 1.2]{CS18}. To see this for $n\geq 2$, use the well-known fact that $K_X+(n+1)A+B$ is nef as $A$ and $B$ are very ample and nef line bundles respectively. write $$L_{l,d}=K_X+(2n+2)A+[(l-1)K_X+(d-2n-2)A+B].$$ Using \cite[Proof of Corollary C]{Ull16} and \cite[Proof of Theorem 1.2]{CS18}, it is enough to show that $d-2n-2\geq (l-1)(n+1)$ which holds as $d\geq (l+1)(n+1)$ by assumption. In other words, the above discussion verifies that \setupref{setup} satisfies \cite[Assumption 1.1]{CS18}.
\end{remark}

We aim to prove the following 

\begin{theorem}\label{pluriqp}
Let $0\leq p\leq n$ be an integer. Suppose we are in the situation of \setupref{setup} and assume $(l,d)$ satisfies \eqref{co3} with $s=p$ when $n\geq 2$. Then $L$ satisfies $(Q_p)$-property.
\end{theorem}

To prove this, first recall that a vector bundle $\mathcal{E}$ on $X$ is called {\it $k$--jet ample} if for every choice of $t$ distinct points $x_1, \cdots, x_t\in X$ and for every tuple $(k_1,\cdots ,k_t)$ of positive integers with $\sum k_i=k+1$, the evaluation map
$$H^0(\mathcal{E})\to H^0\left(\mathcal{E}\otimes \left(\mathcal{O}_X/(\mathcal{I}_{x_1}^{k_1}\otimes\cdots\otimes\mathcal{I}_{x_t}^{k_t})\right)\right)=\bigoplus_{i=1}^{t} H^0\left(\mathcal{E}\otimes \left(\mathcal{O}_X/\mathcal{I}_{x_i}^{k_i}\right)\right)$$
surjects where $\mathcal{I}_{x_i}$ is the ideal sheaf of $x_i\in X$. 

\vspace{5pt}

Recall also that when $n\geq 2$, $A$ is a very ample and $B$ is a nef line bundle on $X$, for a given $l,d\in\mathbb{N}$, we have $$L_{l,d}:=lK_X+dA+B.$$
We first prove a few results on $L_{l,d}$ when $n\geq 2$:

\begin{lemma}\label{l1}
Let $n\geq 2$, $0\leq p\leq n$ and $j\geq \max\left\{\frac{p}{2},1\right\}$. If $l\geq 0$ and $d\geq l(n+1)+2+p$ then $\Omega_X^p(jL_{l,d})$ is $(2j-p)$-jet ample. 
\end{lemma}

\begin{proof} We write $\Omega_X^p(jL_{l,d})=\Omega_X^p\left(2pA\right)\left(jlK_X+(jd-2p)A+jB\right)$. It is well-known that $\Omega_X(2A)$ is globally generated (i.e. 0-jet ample), whence so is $\Omega_X^p(2pA)$. Elementary consideration of Castelnuovo-Mumford regularity and Kodaira vanishing theorem guarantees that $K_X+(n+1)A$ and $K_X+(n+1)A+jB$ are also globally generated. Moreover, since $A$ is a very ample line bundle, it is $1$-jet ample (see \cite[page 3]{BDS99}). Observe that we have $$jd-2p-jl(n+1)\geq 2j-p$$ by our assumption on $d$. Consequently we can write $$jlK_X+(jd-2p)A+jB=[(jl-1)(K_X+(n+1)A)]+[K_X+(n+1)A+jB]+[(2j-p)A]+rA$$ with $r\geq 0$. The conclusion follows from \cite[Proposition 2.3]{BDS99}.\end{proof} 

\begin{lemma}\label{l2}
Let $n\geq 2$, $0\leq p\leq n$. Further assume 
\begin{center}
$l\geq\binom{n-1}{p-1}+1$,\\[5pt]
$d\geq \max\left\{l(n+2), (n+2)\left(l-\binom{n-1}{p-1}-1\right)+2p\left(\binom{n}{p}+1\right), 2p\left(\binom{n}{p}+1\right)+1\right\}.$ 
\end{center}
Then $H^i(\Omega_X^p(jL_{l,d}))=0$ for all $i, j\geq 1$.
\end{lemma}

\begin{proof} It is easy to see using a splitting principle that $\textrm{det}(\Omega_X^p)=\binom{n-1}{p-1}K_X$. Consequently, we obtain $$\Omega_X^p(jL_{l,d})=K_X\otimes\Omega_X^p(2pA)\otimes \textrm{det}(\Omega_X^p(2pA))\otimes Q$$
where $$Q=\underbrace{\left(jl-\binom{n-1}{p-1}-1\right)}_{\text{Term 1}}K_X+\underbrace{\left(jd-2p-\binom{n}{p}2p\right)}_{\text{Term 2}}A+jB.$$
Notice that 
\begin{align*}
jd-2p-\binom{n}{p}2p=& (j-1)d+\left(d-2p-\binom{n}{p}2p\right)\\
\geq & (j-1)l(n+2)+(n+2)\left(l-\binom{n-1}{p-1}-1\right)\\
=& (n+2)\left(jl-\binom{n-1}{p-1}-1\right)
\end{align*}
by assumption. Since $K_X+(n+2)A$ is very ample, we see that $Q$ is ample when $\textrm{Term 1}>0$. Now, $\textrm{Term 1}=0$ is possible only if $j=1$ by our assumption, in which case $\textrm{Term 2}>0$. Thus, under our assumption, $Q$ is ample. Since $\Omega_X^p(2pA)$ is nef, the assertion follows from Griffiths' vanishing theorem (\cite[Variant 7.3.2]{Laz'}).\end{proof}

\begin{proposition}\label{vanishing}
Let $n\geq 2$, $0\leq p\leq n$ and assume $(l,d)$ satisfies \eqref{co3} with $s=p$. 
Then $$H^i(\Omega_{F_x}^{p'}\otimes b_x^*(jL_{l,d})(-2jE_x))=0$$ for all $i,j\geq 1$, $0\leq p'\leq p$ and for all $x\in X$.
\end{proposition}

\begin{proof} Recall from \remarkref{rmkndb} that we can write $L_{l,d}=K_X+(2n+2)A+B'$ where $B'=(l-1)K_X+(d-2n-2)A+B$ is nef. Also recall from \cite[Proof of Proposition 3.2]{CS18} that $$b_x^*(jL_{l,d})(-2jE_x)\cong K_{F_x}+(n+1)P+Q$$ where $P=b_x^*(2A)(-E_x)$ and $Q=(j-1)(K_{F_x}+(n+1)P+b_x^*B')+b_x^*B'$. It is well-known that $K_{F_x}+(n+1)P$ is very ample (see for e.g. \cite[Proof of Proposition 3.2]{CS18} where \cite[page 57]{EL} is used). Since $b_x^*A$ is nef and big, we conclude that 
\begin{equation}\label{ample}
    \textrm{$b_x^*(jL_{l,d})(-2jE_x)$ and $b_x^*(jL_{l,d})(-2jE_x)-K_{F_x}$ are both ample $\forall j\geq 1$, $\forall x\in X$. }
\end{equation}
Thus, the conclusion follows for $p'=0$ by Kodaira vanishing theorem. Henceforth we assume $p'\geq 1$ (whence $p\geq 1$). Since $b_x$ is the blow-up of $X$ at $x$, we observe that $\Omega_{F_x}^1(\log E_x)(-E_x)\cong b_x^*\Omega_X^1$
and we deduce the following exact sequence for any $x\in X$ and for any $1\leq p'\leq p\leq n$, $j\geq 1$
\begin{equation}\label{mainex}
    0\to b_x^*(\Omega_X^{p'}(jL_{l,d}))(-(2j-p'+1)E_x)\to\Omega_{F_x}^{p'}\otimes b_x^*(jL_{l,d})(-2jE_x)\to\Omega_{E_x}^{p'}\otimes b_x^*(jL_{l,d})(-2jE_x)\to 0.
\end{equation}
Notice that \eqref{ample} also yields 
\begin{equation}\label{bott}
    \textrm{$H^i(\Omega_{E_x}^{p'}(b_x^*(jL_{l,d})(-2jE_x)))=0$ for all $i,j\geq 1$ and for all $x\in X$}
\end{equation}
by Bott vanishing theorem, which is known to hold on $E_x$ since it is a projective space. Now assume $2j-p'\geq 0$, and we use the exact sequence 
\begin{equation}\label{jetex}
    0\to \Omega_X^{p'}(jL_{l,d})\otimes\mathcal{I}_x^{2j-p'+1}\to \Omega_X^{p'}(jL_{l,d})\to \Omega_X^{p'}(jL_{l,d})\otimes(\mathcal{O}_X/\mathcal{I}_x^{2j-p'+1})\to 0.
\end{equation}
Using \lemmaref{l1}, we observe that the map $H^0(\Omega_X^{p'}(jL_{l,d}))\to H^0(\Omega_X^{p'}(jL_{l,d})\otimes(\mathcal{O}_X/\mathcal{I}_x^{2j-p'+1}))$ surjects. Passing to the cohomology of \eqref{jetex}, and using \lemmaref{l2}, we get that 
\begin{equation}\label{2j-p}
    \textrm{$H^i(\Omega_X^{p'}(jL_{l,d})\otimes\mathcal{I}_x^{2j-p'+1})=0$ for $i\geq 1$ and for all $x\in X$ if $2j-p'\geq 0$.}
\end{equation}  
It is well-known (see for example \cite[Proof of Lemma 1.4]{BEL}) that for $0\leq s\leq n-1$
\[
R^i{b_x}_*\mathcal{O}_{F_x}(sE_x)= \begin{cases} 
      \mathcal{O}_{F_x} & \textrm{if $i=0$,} \\
      0 & \textrm{if $i>0$.}
      \end{cases}
\]
whence $H^i(b_x^*(\Omega_X^{p'}(jL_{l,d}))(-(2j-p'+1)E_x))=0$ by \lemmaref{l2} for $i\geq 1$ if $2j-p'< 0$. Thus, using \eqref{2j-p} and Leray spectral sequence
we get
\begin{equation}\label{2j-p-1}
    \textrm{$H^i(b_x^*(\Omega_X^{p'}(jL_{l,d}))(-(2j-p'+1)E_x))=0$ for all $i,j\geq 1$ and for all $x\in X$}.
\end{equation} 
The assertion follows from \eqref{bott}, \eqref{2j-p-1} and the cohomology sequence of \eqref{mainex}.\end{proof}

We are now ready to provide the

\begin{proof}[Proof of \theoremref{pluriqp}]
Notice that by assumption, $L$ is 3-very ample by \remarkref{rmkva}, and satisfies $(Q1)$ by \remarkref{rmkndb}.

When $n=1$, $\textrm{deg}(L)\geq 2g+3$ by assumption whence $\textrm{deg}(b_x^*(L)(-2E_x))\geq 2g+1$. Thus $L$ satisfies $(Q2)$ and $(Q3_p)$ for all $p\geq 0$.
Now assume $n\geq 2$ and $L_{l,d}$ satisfies \eqref{co3} with $s=p$. Then $L$ satisfies $(Q2)$ by \eqref{ample}, and satisfies $(Q3_p)$ by \propositionref{vanishing}. 
\end{proof}

\section{Du Bois complex of secant varieties}\label{sectiondb}

Let $X$ be a smooth projective variety of dimension $n$ and $L$ a 3-very ample line bundle. We introduce the following notation that will be used throughout the sequel: $U_X:=\Sigma\backslash X$ and we set $j_{U_X}: U_X\rightarrow\Sigma$ to be the inclusion. 

\vspace{5pt}

In this section, we discuss the Du Bois complex of the secant variety $\Sigma = \Sigma(X,L)$. In particular, we prove \theoremref{thm1} and \theoremref{thm2.3}. 

\subsection{Secant varieties with pre-\texorpdfstring{$p$}{TEXT}-Du Bois singularities.}
We first give sufficient conditions for the secant variety $\Sigma = \Sigma(X,L)$ to have pre-$p$-Du Bois singularities. Recall that this is a condition on the complex $\underline{\Omega}^p_{\Sigma}$. The first result that we are after is a local vanishing statement:

\begin{theorem}\label{lv1}
Let $p\in\mathbb{N}$. If $L$ is $3$-very ample and satisfies $(Q3_p)$, then $$R^it_*\Omega^k_{\mathbb{P}(\mathcal{E}_L)}(\log\Phi)(-\Phi)=0\textrm{ for all } i\geq 1\textrm{ and }0\leq k\leq p.$$
\end{theorem}

\begin{remark}This condition is sufficient for $\cohH^j(\underline{\Omega}^p_{\Sigma}) = 0$ for $j\neq 0$, as is explained in the proof of \theoremref{thm1} (see \eqref{ste}).
\end{remark}

We need some preparations to prove the above result.

\begin{proposition}\label{dbvg}
Let $i,p\geq 1$ be integers and let $L$ be a $3$-very ample line bundle. Suppose that for all $x\in X$, $j\geq 1$, and $0\leq q\leq p$,
 \begin{equation}\label{prophyp}H^i(\Omega^q_{F_x}\otimes b_x^*(jL)(-2jE_x))=0.\end{equation}
 Then $R^it_*\Omega^p_{\mathbb{P}(\mathcal{E}_L)}(\log\Phi)(-\Phi)=0$.
\end{proposition}

\begin{proof} We prove the assertion using the following claims. 

\begin{claim}
    If $H^i(\Omega^p_{\mathbb{P}(\mathcal{E}_L)}(\log\Phi)|_{F_x}\otimes b_x^*(mL)(-2mE_x))=0$ for all $x\in X$ and $m\geq 1$, then $R^{i}t_*\Omega^p_{\mathbb{P}(\mathcal{E}_L)}(\log\Phi)(-\Phi)=0$.
\end{claim}
\begin{proof}
Using the formal function theorem for $x\in \Sigma$, we obtain the isomorphism
\begin{equation*}
    \left(R^it_*\Omega^p_{\mathbb{P}(\mathcal{E}_L)}(\log\Phi)(-\Phi)\right)_x^{\widehat{}}\cong \varprojlim H^i(\Omega^p_{\mathbb{P}(\mathcal{E}_L)}(\log\Phi)(-\Phi)\otimes(\mathcal{O}_{\mathbb{P}(\mathcal{E}_L)}/\mathcal{I}_{F_x}^j)).
\end{equation*}
Since $(R^it_*\Omega^p_{\mathbb{P}(\mathcal{E}_L)}(\log\Phi)(-\Phi))_y=0$ for $y\in U_X$, it is enough to check that for $x\in X$, 
\begin{equation}
    H^i(\Omega^p_{\mathbb{P}(\mathcal{E}_L)}(\log\Phi)(-\Phi)\otimes(\mathcal{O}_{\mathbb{P}(\mathcal{E}_L)}/\mathcal{I}_{F_x}^j))=0\textrm{ for $j\geq 1$.}
\end{equation}
Passing to the cohomology of the following exact sequence 
\begin{equation*}
\begin{split}
    0\to \Omega^p_{\mathbb{P}(\mathcal{E}_L)}(\log\Phi)(-\Phi)\otimes(\mathcal{I}_{F_x}^j/\mathcal{I}_{F_x}^{j+1})\to \Omega^p_{\mathbb{P}(\mathcal{E}_L)}(\log\Phi)(-\Phi)\otimes(\mathcal{O}_{\mathbb{P}(\mathcal{E}_L)}/\mathcal{I}_{F_x}^{j+1})\\
    \to \Omega^p_{\mathbb{P}(\mathcal{E}_L)}(\log\Phi)(-\Phi)\otimes(\mathcal{O}_{\mathbb{P}(\mathcal{E}_L)}/\mathcal{I}_{F_x}^j)\to 0,
\end{split}
\end{equation*} we conclude that it is enough to verify that
\begin{equation*}
    H^i(\Omega^p_{\mathbb{P}(\mathcal{E}_L)}(\log\Phi)(-\Phi)\otimes(\mathcal{I}_{F_x}^j/\mathcal{I}_{F_x}^{j+1}))=0\textrm{ for all $j\geq 0$.}
\end{equation*}
Using \eqref{sym} and \eqref{0}, we conclude that 
\begin{equation*}
\begin{array}{c}
H^i(\Omega^p_{\mathbb{P}(\mathcal{E}_L)}(\log\Phi)(-\Phi)\otimes(\mathcal{I}_{F_x}^j/\mathcal{I}_{F_x}^{j+1}))\\
    \cong \bigoplus\limits_{q=0}^j\left[H^i(\Omega^p_{\mathbb{P}(\mathcal{E}_L)}(\log\Phi)|_{F_x}\otimes b_x^*((q+1)L)(-(2q+2)E_x))\right]^{\oplus \binom{n+j-q-1}{n-1}}
\end{array}
\end{equation*}
and the conclusion follows.
\end{proof}
\begin{claim}\label{cl2}
    Let $m\geq 0$. If $H^i(\Omega^q_{\Phi}|_{F_x}\otimes b_x^*(mL)(-2mE_x))=0$ for $q=p-1,p$ and for all $x\in X$, then $$H^i(\Omega^p_{\mathbb{P}(\mathcal{E}_L)}(\log\Phi)|_{F_x}\otimes b_x^*(mL)(-2mE_x))=0.$$
\end{claim}
\begin{proof}
    Follows by twisting (\ref{ex1}) by $b_x^*(mL)(-2mE_x)$, and passing to cohomology.
\end{proof}

\begin{claim}\label{cl3}
    Assume $m,q\geq 0$. If 
    \begin{equation}\label{modified}
        H^i(\Omega_{F_x}^{q'}\otimes b_x^*(mL)(-2mE_x))=0\,\textrm{ for all }\, 0\leq q'\leq q\,\textrm{ and for all }\, x\in X,
    \end{equation}
    then $H^i(\Omega^q_{\Phi}|_{F_x}\otimes b_x^*(mL)(-2mE_x))=0$ for all $x\in X$.
\end{claim}
\begin{proof}
To show this, we use the following short exact sequence 
\begin{equation}\label{cbs}
    0\to\mathcal{N}_{F_x/\Phi}^*\to \Omega^1_{\Phi}|_{F_x}\to \Omega^1_{F_x}\to 0.
\end{equation}
Since all these sheaves are locally free, there exists a filtration 
\begin{equation}\label{filt}
    \Omega^q_{\Phi}|_{F_x} = F^0 \supseteq F^1 \supseteq \cdots \supseteq F^q \supseteq F^{q+1} = 0\,\textrm{ with $F^l/F^{l+1} \cong \bigwedge^l{\mathcal{N}_{F_x/\Phi}^*} \otimes \Omega^{q-l}_{F_x}.$}
\end{equation} 
We prove by induction that $$H^i(F^l\otimes b_x^*(mL)(-2mE_x))=0\, \textrm{ for $l=0,\ldots , q$.}$$ 
For the base case, we use the short exact sequence $$0 \to F^{q+1} = 0 \to F^q \to \bigwedge^q{\mathcal{N}_{F_x/\Phi}^*}\to 0.$$ We twist the sequence by $b_x^*(mL)(-2mE_x)$ and take cohomology, and then the vanishing follows from the hypothesis (\ref{modified}) (corresponding to $q'=0$) by (\ref{0}) which we recall says $\mathcal{N}_{F_x/\Phi}^*\cong\mathcal{O}_{F_x}^{\oplus n}$. Assume next that we know the result for $l$. We use the short exact sequence 
$$0\to F^l \to F^{l-1} \to \bigwedge^{l-1}{\mathcal{N}_{F_x/\Phi}^*} \otimes \Omega^{q-l+1}_{F_x} \to 0.$$ Twisting by $b_x^*(mL)(-2mE_x)$ and taking cohomology, the result follows from the induction hypothesis, and (\ref{modified}) by (\ref{0}).
\end{proof} 
The assertion of the proposition follows by combining the above three claims.
\end{proof}

\begin{proof}[Proof of \theoremref{lv1}] Note that the statement follows for $p=0$ by \cite[Proof of Proposition 3.2]{CS18}. Thus we assume $p\geq 1$, in which case the conclusion is a consequence of \propositionref{dbvg}. 
\end{proof}

Now we are ready to provide the 

\begin{proof}[Proof of \theoremref{thm1}.] By \cite[Proposition 3.3]{Ste} (see also \cite[\S2.1]{MPOW}), we have an exact triangle:
\begin{equation}\label{ste}
    {\bf R}t_*\Omega^k_{\mathbb{P}(\mathcal{E}_L)}(\log\Phi)(-\Phi)\to \underline{\Omega}^k_{\Sigma}\to \underline{\Omega}^k_{X}\xrightarrow{+1}
\end{equation}
Since $\mathcal{H}^i(\underline{\Omega}^k_X)=0$ for $i\neq 0$ as $X$ is smooth, passing to the cohomology of the above, we obtain $\mathcal{H}^i(\underline{\Omega}^k_{\Sigma})=0$ for $i\neq 0$ as $R^it_*\Omega^k_{\mathbb{P}(\mathcal{E}_L)}(\log\Phi)(-\Phi)=0$ for all $i\geq 1$ and $0\leq k\leq p$ by \theoremref{lv1}.  
\end{proof}

\subsection{Reflexivity condition on \texorpdfstring{$\cohH^0(\underline{\Omega}^p_{\Sigma})$}{TEXT}}
We now aim to describe the sheaf $\cohH^0(\underline{\Omega}^p_{\Sigma})$. Recall that for a seminormal variety with pre-$p$-Du Bois singularities, satisfying the codimension condition on the singular locus, the condition missing for it to have $p$-Du Bois singularities is that on degree zero, the associated graded complex of the Du Bois complex is reflexive. For this, we discuss a reflexivity condition of the push-forward of the sheaf of $p$-forms discussed in \cite{KS}. 

\vspace{5pt}

We work with the standing assumption that $L$ is 3-very ample. We will use, often without stating, that if $L$ satisfies $(Q1)$ the $\Sigma$ is normal (see \cite[Theorem D]{Ull16}), and if $\Sigma$ is normal and $L$ satisfies $(Q3_0)$ then $\Sigma$ has Du Bois singularities (see \cite[Proof of Theorem 1.2 or more precisely Theorem 3.4]{CS18}).

\begin{remark}
We have the isomorphism ${j_{U_X}}_*j_{U_X}^*\Omega_{\Sigma}^{[p]}\cong {j_{U_X}}_*\Omega_{U_X}^p$ for all $p\geq 0$ (recall that by definition $\Omega_{\Sigma}^{[p]}:=(\Omega_{\Sigma}^p)^{**}$). Thus, when $\Sigma$ is normal, we have the isomorphism 
\begin{equation}\label{jux}
    {j_{U_X}}_*\Omega_{U_X}^{p}\cong \Omega_{\Sigma}^{[p]}.
\end{equation}
On the other hand, we have the natural inclusion $$\phi_p:t_*\Omega^p_{\mathbb{P}(\mathcal{E}_L)}\hookrightarrow {j_{U_X}}_*\Omega_{U_X}^{p}.$$ 
Thus, when $\Sigma$ is normal, composing $\phi_p$ with the isomorphism \eqref{jux}, we obtain the maps $$\varphi_p:t_*\Omega_{\mathbb{P}(\mathcal{E}_L)}^p\to \Omega_{\Sigma}^{[p]}.$$
\end{remark}

\begin{proposition}\label{pn1}
The natural inclusion $\phi_p:t_*\Omega^p_{\mathbb{P}(\mathcal{E}_L)}\hookrightarrow{j_{U_X}}_*\Omega_{U_X}^{p}$ is an isomorphism for $0\leq p\leq n-1$.
In particular, if $\Sigma$ is normal,  then $\varphi_p:t_*\Omega_{\mathbb{P}(\mathcal{E}_L)}^p\to \Omega_{\Sigma}^{[p]}$ is an isomorphism for $0\leq p\leq n-1$.
\end{proposition}
\begin{proof} We recall from \cite[(2.3.5)]{KS} that for all $l\geq 0$, Saito's formalism leads to a decomposition 
\begin{equation*}
    {\bf R}t_*\Omega_{\mathbb{P}(\mathcal{E}_L)}^l\cong K_l\oplus R_l  
\end{equation*}
where $K_l,R_l\in{\bf D}^b(\text{Coh}(\Sigma))$. Among other properties, $K_l$ and $R_l$ enjoy the following (see {\it loc. cit.} (2.3.6), (2.3.7)):
\begin{equation}\label{p1}
    \textrm{Supp$(R_l)\subseteq \Sigma_{\textrm{sing}}\subseteq X$,}
\end{equation}
\begin{equation}\label{p2}
    \textrm{$\mathcal{H}^k(K_l)=0$ for $k\geq 2n-l+2$}
\end{equation}
where \eqref{p1} follows from \propositionref{stronglog}. Since $t_*\Omega^p_{\mathbb{P}(\mathcal{E}_L)}$ is torsion-free, \eqref{p1} implies that $t_*\Omega^p_{\mathbb{P}(\mathcal{E}_L)}\cong \mathcal{H}^0(K_p)$. By \cite[(2.3.8)]{KS} we also have the isomorphism $${\bf R}\mathcal{H}\textrm{\textit{om}}_{\mathcal{O}_{\Sigma}}(K_p,\omega_{\Sigma}^{\bullet})\cong K_{2n+1-p}[2n+1].$$ Thus, by \cite[(2.3.9)]{KS}, it is enough to show that 
\begin{equation}\label{needp}
    \dim\left(X\cap \textrm{Supp}\left(\mathcal{H}^k(K_{2n+1-p})\right)\right)\leq 2n-1-k\,\textrm{ for all $k\in \mathbb{Z}$}.
\end{equation}
Observe that $\dim\left(X\cap \textrm{Supp}\left(\mathcal{H}^k(K_{2n+1-p})\right)\right)\leq n$ whence \eqref{needp} holds for $k\leq n-1$. On the other hand, since the dimension of the fibers of $t$ is $\leq n$, we have $R^kt_*\Omega_{\mathbb{P}(\mathcal{E}_L)}^{2n+1-p}=0$ for $k\geq n+1$, whence \eqref{needp} holds if $k\geq n+1$. Finally, since $n\geq p+1$ by assumption, we have $\mathcal{H}^n\left(K_{2n+1-p}\right)=0$ by \eqref{p2}, and the assertion follows. 
\end{proof}

We need one more result in order to prove \theoremref{thm2.3}. From the short exact sequence 
\begin{equation}\label{piff}
    0\to q^*\Omega^1_{X}\to \Omega^1_{\Phi}\to \Omega^1_{\Phi/X}\to 0,
\end{equation} we have an induced morphism $$\gamma_p: \Omega^p_X \to q_*\Omega^p_{\Phi}$$ by taking the wedge product of the first map and then pushing forward. 

\begin{proposition}\label{propgamma} The maps $\gamma_k$ for $k=1,\ldots ,p$ are isomorphisms if and only if $h^0(X,\Omega^k_X)= 0$ for $k=1,\ldots, p$.
\end{proposition}

\begin{proof}
We prove first that given the cohomological conditions, $\gamma_p$ is an isomorphism. For this we use \eqref{piff} again, 
and the fact that all these sheaves are locally free, to obtain a filtration $$\Omega^k_{\Phi} = F^0 \supseteq F^1 \supseteq \cdots \supseteq F^k \supseteq F^{k+1} = 0$$ with quotients $F^l/F^{l+1} \cong q^*\Omega^l_X \otimes \Omega^{k-l}_{\Phi/X}.$ We prove by induction that $q_*F^l \cong \Omega^k_X$. The base case is 
$$0\to F^{k+1} = 0 \to F^k \to q^*\Omega^k_X \to 0,$$ and the claim is clear for $F^k$ by Projection Formula and \cite[Lemma 2.2]{CS18}. Suppose next that $q_*F^l\cong \Omega^k_X$. Consider the short exact sequence $$0 \to F^l \to F^{l-1} \to q^*\Omega^{l-1}_X \otimes \Omega^{k-l+1}_{\Phi/X}\to 0.$$ We pushforward the short exact sequence and by the Projection Formula, we have 
$$0\to \Omega^k_X \to q_*F^{l-1} \to \Omega^{l-1}_X \otimes q_*\Omega^{k-l+1}_{\Phi/X}.$$ Since $h^0(\Omega^{k-l+1}_{F_x}) = h^0(\Omega^{k-l+1}_{X}) = 0$ because $F_x$ is birational to $X$ for all $x\in X$, by Grauert's Theorem (\cite[III, Corollary 12.9]{Har}), $q_*\Omega^{k-l+1}_{\Phi/X} = 0$ and then, $q_*F^{l-1} \cong \Omega^k_X$.

Suppose next that the maps $\gamma_k$ are isomorphisms for $k=1,\dots,p$. 
 We argue by induction. The base case is $p=1$, in which case the assumption says that $\Omega^1_X \cong q_*\Omega^1_{\Phi}$, and therefore, by taking $H^0$ we obtain $h^{1,0}(X) = h^{1,0}(\Phi)$. By \propositionref{iffeq} we have $$h^{1,0}(\Phi) = h^{0,1}(\Phi) = h^{0,1}(X) + h^{0,1}(X).$$ Therefore, $h^{0,1}(X) = h^{0,1}(\Phi) = 0$. This means that $h^0(X,\Omega^1_X)=0$. Suppose now that the result is known for $p=r$, and that $\gamma_k$ are isomorphisms for $k=1,\ldots , r+1$. By the induction hypothesis, $h^0(X,\Omega^k_X)= 0$ for $k=1,\ldots, r$. Moreover, since $\Omega^{r+1}_X \cong q_*\Omega^{r+1}_{\Phi}$, taking $H^0$ we obtain that $h^{r+1,0}(X) = h^{r+1,0}(\Phi).$ By \propositionref{iffeq} and the induction hypothesis, we have $$h^{0,r+1}(\Phi)= h^{0,r+1}(X) + h^{0,r+1}(X).$$ Therefore, $h^{0,r+1}(X) = h^{0,r+1}(\Phi) = 0$. This means that $h^0(X,\Omega^{r+1}_X)=0$.
\end{proof}

Let us describe the maps $\delta_p$ in detail. The functoriality of Du Bois complexes induces a canonical map $\underline{\Omega}^p_{\Sigma}\to{\bf R}t_*\underline{\Omega}^p_{\mathbb{P}(\mathcal{E}_L)}={\bf R}t_*\Omega_{\mathbb{P}(\mathcal{E}_L)}^p$, which yields $\beta_p:\mathcal{H}^0(\underline{\Omega}^p_{\Sigma})\to t_*\Omega_{\mathbb{P}(\mathcal{E}_L)}^p.$ In particular, when $\Sigma$ is normal we obtain for all $p\geq 0$, the natural map 
\begin{equation}\label{defdelk}
    \delta_p:=\varphi_p\circ\beta_p:\mathcal{H}^0(\underline{\Omega}^p_{\Sigma})\to\Omega^{[p]}_{\Sigma}.
\end{equation}

We record a fact that we will use without any further reference.

\begin{remark}
Assume $\Sigma$ is normal. The map $\delta_k$ is an isomorphism if and only if $\beta_k$ and $\varphi_k$ are both isomorphisms. Indeed, this follows immediately from the injectivity of $\varphi_k$.
\end{remark}

We are now ready to provide the 

\begin{proof}[Proof of \theoremref{thm2.3}]
Recall that $\Sigma$ is normal and has Du Bois singularities. In particular, we have the isomorphisms 
\begin{equation}\label{nor}
    \Omega_{\Sigma}^{[2n+1]}\cong {j_{U_X}}_*\omega_{U_X}\cong \omega_{\Sigma}:=\mathcal{H}^{-(2n+1)}\omega_{\Sigma}^{\bullet}.
\end{equation}
Also recall that by \theoremref{lv1} 
\begin{equation}\label{adhoc p}
    R^1t_*\Omega^k_{\mathbb{P}(\mathcal{E}_L)}(\log\Phi)(-\Phi)=0\,\textrm{ for all }\, 0\leq k\leq p.
\end{equation}
We work with the following commutative diagram with exact rows: 
\begin{equation}\label{K}
    \begin{tikzcd}
    0\arrow[r] & t_*\Omega^k_{\mathbb{P}(\mathcal{E}_L)}(\log\Phi)(-\Phi)\arrow[r]\arrow[d, equal] & \mathcal{H}^0(\underline{\Omega}^k_{\Sigma})\arrow[r]\arrow[d, "\beta_k"] & \Omega^k_X\arrow[r]\arrow[d, "\gamma_k"] & 0\\
    0\arrow[r] & t_*\Omega^k_{\mathbb{P}(\mathcal{E}_L)}(\log\Phi)(-\Phi)\arrow[r] & t_*\Omega^k_{\mathbb{P}(\mathcal{E}_L)}\arrow[r] & q_*\Omega^k_{\Phi}\arrow[r] & 0
\end{tikzcd}
\end{equation}
where the top sequence is obtained by passing to the cohomology of \eqref{ste}, the bottom row is obtained by taking the direct images of the sequence $$0\to \Omega^k_{\mathbb{P}(\mathcal{E}_L)}(\log\Phi)(-\Phi)\to  \Omega^k_{\mathbb{P}(\mathcal{E}_L)}\to \Omega^k_{\Phi}\to 0.$$ Both rows are exact on the right because of \eqref{adhoc p}. 

First assume $H^0(\Omega_X^k)=0$ for $1\leq k\leq p$. Then, by Proposition \propositionref{propgamma}, $\gamma_k$ is an isomorphism for all $1\leq k\leq p$ whence $\beta_k$'s are isomorphisms in the same range. If $p\leq n-1$, then the conclusion follows since $\varphi_k:t_*\Omega_{\mathbb{P}(\mathcal{E}_L)}^k\to \Omega_{\Sigma}^{[k]}$ are isomorphism by \propositionref{pn1}. If $p\geq n$, then $H^n(\mathcal{O}_X)=0$ by assumption, whence $t_*\omega_{\mathbb{P}(\mathcal{E}_L)}\cong \omega_{\Sigma}$ by \cite[Theorem 5.8]{CS18}. Thus, the conclusion in this case follows by \cite[Theorem 1.4]{KS}. Conversely, assume $\delta_k:\mathcal{H}^0(\underline{\Omega}^k_{\Sigma})\to\Omega^{[k]}_{\Sigma}$ are isomorphisms for $1\leq k\leq p$, whence $\beta_k$'s are isomorphisms. Now the conclusion follows again from \propositionref{propgamma}. This completes the proof.
\end{proof}

\begin{proof}[Proof of \corollaryref{corpdb}] Under our assumptions, $\Sigma$ has pre-$p$-Du Bois singularities by \theoremref{thm1} and satisfies the codimension condition as $p\leq \lfloor\frac{n}{2}\rfloor$. Also recall that $\Sigma$ is normal and has Du Bois singularities, and in particular is seminormal. We may assume that $n\geq 2$, $p\geq 1$. 
Note that $\mathcal{H}^0(\underline{\Omega}_{\Sigma}^k)$ is reflexive for $1\leq k\leq p$ if and only if 
$\delta_k$'s are isomorphisms for $1\leq k\leq p$. The conclusion now follows from \theoremref{thm2.3}.
\end{proof}

\begin{example}[Curves embedded by line bundles of smaller degree]\label{cliff}
Suppose $X$ is a curve and $L$ is a non-special (i.e., $H^1(L)=0$) $3$-very ample line bundle on $X$. Note that $L$ satisfies $(Q3_p)$ for all $p\geq 0$ by assumption. In this case:
\begin{itemize}
    \item[(i)] Our proof shows that the singularities of $\Sigma$ are pre-$p$-Du Bois for all $p\geq 0$. 
    \item[(ii)] \cite[Proof of Theorem 3.4]{CS18} shows that the singularities of $\Sigma$ are Du Bois if $\Sigma$ is normal.
    \item[(iii)] Recall that the main result of \cite{Ull16} asserts that $\Sigma$ is normal if $L(-2x)$ is projectively normal for $x\in X$. In particular, via \cite[Theorem 1]{GL} (see \cite[Proof of Corollary B]{Ull16}), if one of the following holds:
    \begin{enumerate}
        \item $\textrm{deg}(L)=2g+1$ and $\textrm{Cliff}(X)\geq 2$ (equivalently, $X$ is not hyperelliptic, trigonal, or plane quintic); or
        \item $\textrm{deg}(L)=2g+1$ and $\textrm{Cliff}(X)\geq 1$ (equivalently $X$ is not hyperelliptic),
    \end{enumerate}
    then $\Sigma$ is normal and has Du Bois singularities.
\end{itemize}
On a complementary direction, although the secant varieties of canonical curves with $\textrm{Cliff}(X)\geq 3$ are normal by \cite[Corollary B]{Ull16}, it was shown in \cite{CS18} that in this case the singularities of $\Sigma$ are not Du Bois.
\end{example}

\subsection{Further results} We can actually say more about the associated graded pieces of the Du Bois complex (under the standing assumption of 3-very ampleness of $L$) that are usually not considered in the definition of higher Du Bois singularities. In particular, we aim to prove the following: 

\begin{theorem}\label{thm2} Let $\delta_k:\mathcal{H}^0(\underline{\Omega}^k_{\Sigma})\to \Omega_{\Sigma}^{[k]}$ be the natural maps defined in \eqref{defdelk}. Then the following statements hold:
\begin{enumerate}
    \item Assume $L$ satisfies $(Q_n)$-property. Then the following are equivalent:
    \begin{itemize}
        \item[(i)] $H^n(\mathcal{O}_X)=0$,
        \item[(ii)] $\delta_{2n+1}$ is an isomorphism,
        \item[(iii)] $\delta_{2n}$ is an isomorphism.
    \end{itemize}
    \item Assume $n\geq 3$ and $L$ satisfies $(Q_n)$-property. Let $n+2\leq p\leq 2n-1$. If $H^k(\mathcal{O}_X)=0$ for all $k\geq p-n-1$, then $\delta_p$ is an isomorphism.\footnote{When $L$ is 3-very ample, the conclusions of \theoremref{thm2} are also valid if $L$ satisfies $(Q3_n)$ and $\Sigma$ is normal.}
\end{enumerate}
\end{theorem}

In order to prove part (2) of the above theorem, we need another local vanishing result:

\begin{proposition}\label{prop1}
Let $i,p\geq 1$ be integers and let $L$ be a $3$-very ample line bundle. Suppose that for all $x\in X$, $j\geq 0$, and $0\leq q\leq p$,
\begin{equation}\label{onlylog}
    H^i(\Omega_{F_x}^q\otimes b_x^*(jL)(-2jE_x))=0.
\end{equation}
Then $R^{i}t_*\Omega^p_{\mathbb{P}(\mathcal{E}_L)}(\log\Phi)=0$. 
\end{proposition}

\begin{proof} As we did in the proof of \propositionref{dbvg}, we first prove the following 

\begin{claim}\label{cl1}
    If $H^i(\Omega^p_{\mathbb{P}(\mathcal{E}_L)}(\log\Phi)|_{F_x}\otimes b_x^*(mL)(-2mE_x))$ for all $x\in X$ and for all $m\geq 0$ then $R^it_*\Omega^p_{\mathbb{P}(\mathcal{E}_L)}(\log\Phi)=0$.
\end{claim}
\begin{proof}
Clearly $(R^it_*\Omega^p_{\mathbb{P}(\mathcal{E}_L)}(\log\Phi))_y=0$ if $y\in\Sigma$ and $y\notin X$. For $x\in X$, we use the formal function theorem 
\begin{equation*}
    \left(R^it_*\Omega^p_{\mathbb{P}(\mathcal{E}_L)}(\log\Phi)\right)_x^{\widehat{}}\cong \varprojlim H^i(\Omega^p_{\mathbb{P}(\mathcal{E}_L)}(\log\Phi)\otimes(\mathcal{O}_{\mathbb{P}(\mathcal{E}_L)}/\mathcal{I}_{F_x}^j)).
\end{equation*}
Thus, in order to prove the assertion, it is enough to show that 
\begin{equation}\label{ratv1}
    H^i(\Omega^p_{\mathbb{P}(\mathcal{E}_L)}(\log\Phi)\otimes(\mathcal{O}_{\mathbb{P}(\mathcal{E}_L)}/\mathcal{I}_{F_x}^j))=0\textrm{ for all $j\geq 1$}.
\end{equation}
Passing to the cohomology of the exact sequence
\begin{equation}\label{indratv}
\begin{split}
    0\to\Omega^p_{\mathbb{P}(\mathcal{E}_L)}(\log\Phi)\otimes(\mathcal{I}_{F_x}^j/\mathcal{I}_{F_x}^{j+1})\to\Omega^p_{\mathbb{P}(\mathcal{E}_L)}(\log\Phi)\otimes(\mathcal{O}_{\mathbb{P}(\mathcal{E}_L)}/\mathcal{I}_{F_x}^{j+1})\\
    \to \Omega^p_{\mathbb{P}(\mathcal{E}_L)}(\log\Phi)\otimes(\mathcal{O}_{\mathbb{P}(\mathcal{E}_L)}/\mathcal{I}_{F_x}^{j})\to 0,
\end{split}
\end{equation}
we conclude that to prove \eqref{ratv1}, it is enough to show that 
\begin{equation}\label{needratv}
    H^i(\Omega^p_{\mathbb{P}(\mathcal{E}_L)}(\log\Phi)\otimes(\mathcal{I}_{F_x}^j/\mathcal{I}_{F_x}^{j+1}))=0\textrm{ for all $j\geq 0$.}
\end{equation}
Observe that by \eqref{sym}, we have 
\begin{equation}\label{sym'}
    H^i(\Omega^p_{\mathbb{P}(\mathcal{E}_L)}(\log\Phi)\otimes(\mathcal{I}_{F_x}^j/\mathcal{I}_{F_x}^{j+1}))\cong \bigoplus\limits_{m=0}^j\left[H^i(\Omega^p_{\mathbb{P}(\mathcal{E}_L)}(\log\Phi)|_{F_x}\otimes b_x^*(mL)(-2mE_x))\right]^{\oplus \binom{n+j-m-1}{n-1}}
\end{equation}
and the conclusion follows.
\end{proof}
We now invoke \claimref{cl2} and \claimref{cl3}, which immediately completes the proof.
\end{proof}

The previous result can be used to obtain one more reflexivity statement:

\begin{proposition}\label{toptob}
Let $n\geq 3$ and $n+2\leq p\leq 2n-1$. Assume the following conditions:
\begin{itemize}
\item[(i)] $H^n(\Omega_{F_x}^q\otimes b_x^*(jL)(-jE_x))=0 \,\textrm{ for all }\, j\geq 1$, $0\leq q\leq 2n+1-p$,
\item[(ii)] $H^0(\Omega_{X}^{n-q})=0\,\textrm{ for all }\, 0\leq q\leq 2n+1-p$.
\end{itemize} 
Then the two natural maps $$t_*\Omega_{\mathbb{P}(\mathcal{E}_L)}^p(\log\Phi)(-\Phi)\hookrightarrow t_*\Omega_{\mathbb{P}(\mathcal{E}_L)}^p\hookrightarrow{j_{U_X}}_*\Omega_{U_X}^p$$ are isomorphisms. In particular, if (i) and (ii) hold and if $\Sigma$ is normal, then the map $$t_*\Omega_{\mathbb{P}(\mathcal{E}_L)}^p(\log\Phi)(-\Phi)\to\Omega_{\Sigma}^{[p]}$$ is an isomorphism. 
\end{proposition}

\begin{proof} We first note the following isomorphisms
\begin{equation*}
\begin{split}
    {\bf R}\mathcal{H}\textit{om}_{\mathcal{O}_{\Sigma}}\left({\bf R}t_*\Omega_{\mathbb{P}(\mathcal{E}_L)}^p(\log\Phi)(-\Phi),\omega_{\Sigma}^{\bullet}\right) & \cong {\bf R}t_*{\bf R}\mathcal{H}\textit{om}_{\mathcal{O}_{\mathbb{P}(\mathcal{E}_L)}}\left(\Omega_{\mathbb{P}(\mathcal{E}_L)}^p(\log\Phi)(-\Phi),\omega_{\mathbb{P}(\mathcal{E}_L)}[2n+1]\right)\\
     & \cong {\bf R}t_*\Omega_{\mathbb{P}(\mathcal{E}_L)}^{2n+1-p}(\log\Phi)[2n+1],
\end{split}
\end{equation*} where the first one is obtained via duality.
Observe that it is enough to show that the natural map $$t_*\Omega_{\mathbb{P}(\mathcal{E}_L)}^p(\log\Phi)(-\Phi)\hookrightarrow {j_{U_X}}_*\Omega_{U_X}^p$$ is an isomorphism. Thus, by \cite[Proposition 6.4]{KS}, it is enough to show that 
\begin{equation}\label{extn}
    \dim\left(X\cap\textrm{Supp}\left(R^k\Omega_{\mathbb{P}(\mathcal{E}_L)}^{2n+1-p}(\log\Phi)\right)\right)\leq 2n-1-k\,\textrm{ for all }\, k\in\mathbb{Z}.
\end{equation}
As in the proof of \propositionref{pn1}, \eqref{extn} holds if $k\leq n-1$ or if $k\geq n+1$. Thus it is enough to show that 
\begin{equation}\label{needmore}
    R^n\Omega_{\mathbb{P}(\mathcal{E}_L)}^{2n+1-p}(\log\Phi)=0.
\end{equation} Notice that \begin{equation}\label{more}
    H^n(\Omega_{F_x}^q\otimes b_x^*(jL)(-jE_x))=0 \,\textrm{ for all }\, j\geq 0, 0\leq q\leq 2n+1-p.
\end{equation}
Indeed \eqref{more} follows from hypothesis (i) for $j\geq 1$. For $j=0$, the vanishing follows from hypothesis (ii) as we have $h^n(\Omega_{F_x}^q)= 
h^0(\Omega_{F_x}^{n-q})=h^0(\Omega_X^{n-q}).$ 
Now \eqref{needmore} follows from \propositionref{prop1}.
\end{proof}

\begin{proof}[Proof of \theoremref{thm2}.] 
Recall that under our assumption, $\Sigma$ is normal and has Du Bois singularities. We also recall that $\delta_k$ is an isomorphism if and only if $\beta_k$ and $\varphi_k$ are isomorphisms.

We first prove (1) and set $p=2n+1$. It follows from the middle column of \eqref{K} that $$\mathcal{H}^0(\underline{\Omega}^{2n+1}_{\Sigma})\cong t_*\omega_{\mathbb{P}(\mathcal{E}_L)}.$$ Thus, using \eqref{nor}, we see that the equivalence of (i) and (ii) is a consequence of \cite[Theorem 5.8]{CS18}. We now show the equivalence of (i) and (iii). If (i) holds, then as before, we see that the inclusion $t_*\omega_{\mathbb{P}(\mathcal{E}_L)}\hookrightarrow{j_{U_X}}_*\omega_{U_X}\cong \omega_{\Sigma}$ is an isomorphism by \cite[Theorem 5.8]{CS18}. Thus, by \cite[Theorem 1.5]{KS}, we see that the two morphisms $$t_*\Omega^{2n}_{\mathbb{P}(\mathcal{E}_L)}(\log\Phi)(-\Phi)\hookrightarrow t_*\Omega^{2n}_{\mathbb{P}(\mathcal{E}_L)}\hookrightarrow{j_{U_X}}_*\Omega_{U_X}^{2n}$$ are both isomorphisms. Consequently, it follows from \eqref{K} that $\beta_{2n}$ and $\varphi_{2n}$ are both isomorphisms. Conversely, assume (iii) holds. Then \eqref{K} implies $H^0(\omega_{\Phi})=0$ whence by \propositionref{iffeq} we get $H^0(\omega_X)=0$ which is (i).

Now we prove (2). Since $p\geq n+2$ by assumption, it follows from the first row of \eqref{K} that $$t_*\Omega_{\mathbb{P}(\mathcal{E}_L)}^p(\log\Phi)(-\Phi)\cong\mathcal{H}^0(\underline{\Omega}^p_{\Sigma}).$$ The conclusion now follows from \propositionref{toptob}.
\end{proof}


\begin{remark}
If the natural map $\delta_k:\mathcal{H}^0(\underline{\Omega}^k_{\Sigma})\to \Omega_{\Sigma}^{[k]}$ is an isomorphisms for some $k\geq 1$, then $$\min\left\{h^0(\Omega_X^{k-i}),h^0(\Omega_X^i)\right\}=0\,\textrm{ for all $0\leq i\leq k$}.$$
Indeed, since $\delta_k$ is an isomorphism, $\beta_k$ is also an isomorphism. Consequently, from \eqref{K} we see that $\gamma_k$ is an isomorphism, whence $h^0(\Omega_X^k)=h^0(\Omega_{\Phi}^k)$. The conclusion now follows from \propositionref{iffeq}.    
\end{remark}

\begin{corollary}\label{graded}
    Let $p\in\mathbb{N}$. If one of the following holds:
    \begin{itemize}
    \item[(i)] $L$ satisfies $(Q_p)$-property and $H^k(\mathcal{O}_{X})=0$ for $1\leq k\leq p$;
    \item[(ii)] $n\geq 2$, $L$ satisfies $(Q_n)$-property, $p\in\left\{2n, 2n+1\right\}$, and $H^n(\mathcal{O}_X)=0$;
    \item[(iii)] $n\geq 3$, $L$ satisfies $(Q_n)$-property, $n+2\leq p\leq 2n-1$, and $H^{k}(\mathcal{O}_X)=0$ for all $k\geq p-n-1$;
    \end{itemize}
    then there is a natural quasi-isomorphism $\underline{\Omega}^p_{\Sigma}\cong\Omega_{\Sigma}^{[p]}$.
\end{corollary}

\begin{remark}
    Notice that the condition on $k$ in (i) is vacuous when $p=0$.
\end{remark}

\begin{proof} This is an immediate consequence of \theoremref{thm1}, \theoremref{thm2.3} and \theoremref{thm2}.\end{proof}

\corollaryref{nakano} and \corollaryref{corhdiff} are special cases of the following more general results:

\begin{corollary}\label{nakano'}
    Let $p\in\mathbb{N}$ and let $\mathcal{L}$ be an ample line bundle on $\Sigma$. If one of the conditions (i), (ii) or (iii) of \corollaryref{graded} holds, then $$H^q(\Omega_{\Sigma}^{[p]}\otimes \mathcal{L})=0\,\textrm{ when $p+q>\dim\Sigma=2n+1$}.$$
\end{corollary}

\begin{proof}
    This is an immediate consequence of \corollaryref{graded} and \cite[Theorem V.5.1]{GN}.
\end{proof}

\begin{corollary}[Description of the $h$-differentials]\label{h'}
    Let $p\in\mathbb{N}$. If one of the conditions (i), (ii) or (iii) of \corollaryref{graded} holds, then there is a natural isomorphism $\Omega^p_h|_{\Sigma}\cong\Omega_{\Sigma}^{[p]}$.
\end{corollary}  

\begin{proof}
    This is an immediate consequence of \corollaryref{graded} and \cite[Theorem 7.12]{HJ}.
\end{proof}

\section{Higher rational singularities of secant varieties}\label{sectionrat}

As before, it is our standing assumption that $X$ is a smooth projective variety and $L$ is a 3-very ample line bundle on $X$. Here we study the dual ${\bf D}_{\Sigma}(\underline{\Omega}_{\Sigma}^{p})$. In particular, we prove \theoremref{prerat}.

\subsection{Cohomology of the dual of \texorpdfstring{$\underline{\Omega}_{\Sigma}^{p}$}{TEXT}}
Recall that for a variety $Z$, we set $${\bf D}_Z(-):={\bf R}\mathcal{H}\textit{om}_{\mathcal{O}_Z}(-,\omega_Z^{\bullet})[-\dim Z].$$
If $Z\subset W$ is of codimension $c$, we have ${\bf D}_{W}(-)\cong{\bf D}_{Z}(-)[-c]$ for complexes of $\mathcal{O}_Z$-modules.

\begin{remark}\label{alex}
Since $L$ is 3-very ample, we have \eqref{ste}, dualizing which (with obvious modifications of wedge powers) we obtain the exact triangle:
\begin{equation}\label{dualgro}
    {\bf D}_{\Sigma}(\underline{\Omega}_X^{2n+1-k})\to {\bf D}_{\Sigma}(\underline{\Omega}_{\Sigma}^{2n+1-k})\to {\bf R}t_*\Omega_{\mathbb{P}(\mathcal{E}_L)}^k(\log\Phi)\xrightarrow{+1}.
\end{equation}
Thus, we have the isomorphisms ${\bf D}_{\Sigma}(\underline{\Omega}_{\Sigma}^{2n+1-k})\cong {\bf R}t_*\Omega_{\mathbb{P}(\mathcal{E}_L)}^k(\log\Phi)$ for $0\leq k\leq n$; in particular
\begin{equation*}
    \mathcal{H}^i({\bf D}_{\Sigma}(\underline{\Omega}_{\Sigma}^{2n+1-k}))\cong R^it_*\Omega_{\mathbb{P}(\mathcal{E}_L)}^k(\log\Phi)\,\textrm{ for all }\, 0\leq k\leq n, i\geq 0.
\end{equation*}
\end{remark}

We first have the following results for the cohomology of the dual complex:

\begin{lemma}\label{h0}
The following assertions hold for all $0\leq k\leq 2n+1$:
\begin{enumerate}
    \item $\mathcal{H}^i({\bf D}_{\Sigma}(\underline{\Omega}_{\Sigma}^{2n+1-k}))\cong R^it_*\Omega_{\mathbb{P}(\mathcal{E}_L)}^k(\log\Phi)\,\textrm{ for all }\, 0\leq i\leq n-1$.
    \item If $i\geq n+2$, then 
    $\mathcal{H}^i({\bf D}_{\Sigma}(\underline{\Omega}_{\Sigma}^{2n+1-k}))=0$.
\end{enumerate}
\end{lemma}

\begin{proof} We start by observing the following consequence of the smoothness of $X$: 
\begin{equation}\label{hid1}
{\bf D}_{\Sigma}(\underline{\Omega}_X^{2n+1-k})\cong {\bf D}_X(\underline{\Omega}_X^{2n+1-k})[-n-1]\cong \begin{cases} 
      \Omega_X^{k-n-1}[-n-1] & \textrm{if }\, k\geq n+1; \\
      0 & \textrm{otherwise}.
   \end{cases}
\end{equation}   
Thus
\begin{equation}\label{hid}
    \mathcal{H}^i({\bf D}_{\Sigma}(\underline{\Omega}_{X}^{2n+1-k}))\cong \begin{cases} 
      \Omega_X^{k-n-1} & \textrm{if }\, i=n+1\, \textrm{ and }\, k\geq n+1; \\
      0 & \textrm{otherwise}.
   \end{cases}
\end{equation}
Consequently, we obtain (1) from \eqref{dualgro}. Further, $R^it_*\Omega_{\mathbb{P}(\mathcal{E}_L)}^k(\log\Phi)=0$ for $i\geq n+1$, whence by \eqref{dualgro} we conclude that for all $0\leq k\leq 2n-1$, we have $$\mathcal{H}^i({\bf D}_{\Sigma}(\underline{\Omega}_X^{2n+1-k}))\cong\mathcal{H}^i({\bf D}_{\Sigma}(\underline{\Omega}_{\Sigma}^{2n+1-k}))\,\textrm{ for all }\, i\geq n+2.$$
The conclusions now follow from \eqref{hid}.
\end{proof}

We end this section by showing that we can say more about $\mathcal{H}^0(\tau_k)$, a fact that will not be used in the sequel. Recall that we have the natural map $$\tau_k:\underline{\Omega}_{\Sigma}^k\to {\bf D}_{\Sigma}(\underline{\Omega}_{\Sigma}^{2n+1-k}).$$ The induced map $\mathcal{H}^0(\tau_k)$ via the isomorphism of \lemmaref{h0} (1) can be identified with $$\phi'_k\circ\beta_k: \mathcal{H}^0(\underline{\Omega}_{\Sigma}^k)\to t_*\Omega_{\mathbb{P}(\mathcal{E}_L)}^k(\log\Phi)$$ where $\phi_k:t_*\Omega^k_{\mathbb{P}(\mathcal{E}_L)}\hookrightarrow{j_{U_X}}_*\Omega_{U_X}^k$ is the composition of the two inclusions $\phi'_k$, $\phi''_k$ described in the following diagram:
\begin{equation*}
    \begin{tikzcd}
    t_*\Omega^k_{\mathbb{P}(\mathcal{E}_L)}\arrow[r, hook,"\phi'_k"]\arrow[rr, swap, "\phi_k", bend right=18] & t_*\Omega^k_{\mathbb{P}(\mathcal{E}_L)}(\log\Phi)\arrow[r, hook, "\phi''_k"] & {j_{U_X}}_*\Omega_{U_X}^k
\end{tikzcd}
\end{equation*}
In particular, when $\Sigma$ is normal, the maps $\delta_k:\mathcal{H}^0(\underline{\Omega}^k_{\Sigma})\to\Omega^{[k]}_{\Sigma}$ in \eqref{defdelk} is the composition of $\mathcal{H}^0(\tau_k)$ and an inclusion.

\begin{remark}\label{comp}
    We observe that
    \begin{itemize}
        \item[(a)] $\mathcal{H}^0(\tau_k)$ is an isomorphism if and only if $\beta_k$ and $\phi'_k$ are isomorphisms.
    \item[(b)] $\phi_k$ is an isomorphism if and only if $\phi'_k$ and $\phi''_k$ are isomorphisms.
    \end{itemize} 
    Also recall that if $\Sigma$ is normal, then $\delta_k$ is an isomorphism if and only if $\beta_k$ and $\phi_k$ are isomorphisms.
    In particular, if $\Sigma$ is normal and $\delta_k$ is an isomorphism, then $\mathcal{H}^0(\tau_k)$ is also an isomorphism.
\end{remark}

We have the following consequence of \theoremref{thm2.3} and \theoremref{thm2}:

\begin{proposition}\label{hod}
    Let $\mathcal{H}^0(\tau_k): \mathcal{H}^0(\underline{\Omega}_{\Sigma}^k)\to \mathcal{H}^0({\bf D}_{\Sigma}(\underline{\Omega}_{\Sigma}^{2n+1-k}))$ be the natural maps. Then the following hold:
    \begin{itemize}
        \item[(1)] Let $0\leq p\leq 2n+1$, and assume $L$ satisfies $(Q_p)$-property. Then $\mathcal{H}^0(\tau_{k})$'s are isomorphisms for all $0\leq k\leq p$ if and only if $H^k(\mathcal{O}_X)=0$ for all $1\leq k\leq p$.\footnote{Notice that this condition is vacuous when $p=0$.}
        \item[(2)] Assume $n\geq 2$ and $L$ satisfies $(Q_n)$-property. If $H^n(\mathcal{O}_X)=0$ then $\mathcal{H}^0(\tau_{2n+1})$ and $\mathcal{H}^0(\tau_{2n})$ are isomorphisms. 
        \item[(3)] Assume $n\geq 3$ and $L$ satisfies $(Q_n)$-property. Let $n+2\leq p\leq 2n-1$. If $H^{n-k}(\mathcal{O}_X)=0$ for all $0\leq k\leq 2n+1-p$, then $\mathcal{H}^0(\tau_{p})$ is an isomorphism.
    \end{itemize}
\end{proposition}

\begin{proof} We first recall that $\Sigma$ is normal. 
Notice that (2) and (3) follow by \theoremref{thm2} (1), (2) and \remarkref{comp}.  

We now prove (1). Recall that we have the diagram \eqref{K} with exact rows for all $0\leq k\leq p$. Since $q_*\mathcal{O}_{\Phi}\cong\mathcal{O}_X$ by \eqref{cs}, it follows that $\beta_0$ is an isomorphism. Also, $\phi_0$ is an isomorphism as $\Sigma$ is normal whence $\phi'_0$ is an isomorphism by \remarkref{comp} (b). Thus $\mathcal{H}^0(\tau_0)$ is an isomorphism by \remarkref{comp} (a). 

We now assume $k\geq 1$, hence $p\geq 1$. If $\mathcal{H}^0(\tau_k)$ is an isomorphism for all $1\leq k\leq p$, then $\beta_k$'s whence $\gamma_k$'s are isomorphisms in the same range. Thus $H^0(\Omega_X^k)=0$ for all $1\leq k\leq p$ by \propositionref{propgamma}. The converse follows from \theoremref{thm2.3} and \remarkref{comp}. \end{proof}

\subsection{Secant varieties with pre-\texorpdfstring{$1$}{TEXT}-rational singularities.} We now prove \theoremref{prerat}. 

\begin{proof}[Proof of \theoremref{prerat}.] 
Observe that by \remarkref{alex} 
\begin{equation}\label{ander}
    \mathcal{H}^i({\bf D}_{\Sigma}(\underline{\Omega}_{\Sigma}^{2n+1-k}))\cong R^it_*\Omega_{\mathbb{P}(\mathcal{E}_L)}^k(\log\Phi)\,\textrm{ for all }\, 0\leq k\leq 1, i\geq 0.
\end{equation}

Assume (1) holds. Then by \eqref{ander}, we have $R^1t_*\Omega_{\mathbb{P}(\mathcal{E}_L)}^1(\log\Phi)=0$. Using the restriction sequence and \theoremref{lv1}, we obtain that $$R^1t_*\Omega_{\mathbb{P}(\mathcal{E}_L)}^1(\log\Phi)\cong R^1q_*\Omega_{\mathbb{P}(\mathcal{E}_L)}^1(\log\Phi)|_{\Phi}=0.$$ The short exact sequence $$0\to\Omega_{\Phi}^1\to \Omega_{\mathbb{P}(\mathcal{E}_L)}^1(\log\Phi)|_{\Phi}\to\mathcal{O}_{\Phi}\to 0$$ along with the vanishings above implies that the resulting map $q_*\mathcal{O}_{\Phi}\to R^1q_*\Omega_{\Phi}^1$ is surjective. But $$\textrm{rank}(q_*\mathcal{O}_{\Phi})=1\,\textrm{ and }\, \textrm{rank}(R^1q_*\Omega_{\Phi}^1)=\begin{cases}
    nh^{1,0}(X)+h^{1,1}(X)+1 & \textrm{ if } n\geq 2;\\
    nh^{1,0}(X)+h^{1,1}(X) & \textrm{ if } n=1
\end{cases}$$
where the second equality follows from \lemmaref{ext}. Thus, we conclude that $n=1$ and $(X,L)$ is a rational normal curve of degree $\geq 3$.

Now we prove the converse and assume $(X,L)$ is a rational normal curve of degree $\geq 3$. 
We claim that 
\begin{equation}\label{ratneed}
    R^it_*\Omega_{\mathbb{P}(\mathcal{E}_L)}^k(\log\Phi)\,\textrm{ for all }\, 0\leq k\leq 1, i\geq 1.
\end{equation} Notice that \eqref{ratneed} holds for $k=0$ by \cite[Corollary 1.5]{CS18}. Since \eqref{ratneed} holds for $i\geq 2$, it is enough to show \eqref{ratneed} for $i=k=1$. We first show that \begin{equation}\label{ratneed'}
    H^1(\Omega^1_{\mathbb{P}(\mathcal{E}_L)}(\log\Phi)|_{F_x})=0\,\textrm{ for all }\, x\in X.
\end{equation} 
Notice that in this case, $F_x\cong\mathbb{P}^1$, $\Phi\cong\mathbb{P}^1\times\mathbb{P}^1$ and one sees easily that $h^1(\Omega^1_{\Phi}|_{F_x})=h^1(\Omega_{F_x}^1)=1$. Passing to the cohomology of the exact sequence 
\begin{equation}\label{ex1p1}
    0\to{\Omega}^1_{\Phi}|_{F_x}\to \Omega^1_{\mathbb{P}(\mathcal{E}_L)}(\log\Phi)|_{F_x}\to\mathcal{O}_{F_x}\to 0,
\end{equation}
 and by \eqref{cbs}, we obtain the composite map 
\begin{equation}\label{connecting}
    H^0(\mathcal{O}_{F_x})\to H^1(\Omega_{\Phi}^1|_{F_x})\to H^1(\Omega_{F_x}^1)
\end{equation}
Now we observe that the composition of the maps \eqref{connecting} sends $1\in H^0(\mathcal{O}_{F_x})$ to the cohomology class $c_1(\mathcal{O}_{\mathbb{P}(\mathcal{E}_L)}(\Phi)|_{F_x})\in H^1(\Omega_{F_x}^1)$ (see the diagram \eqref{omega}), whence it is injective since $c_1(\mathcal{O}_{\mathbb{P}(\mathcal{E}_L)}(\Phi)|_{F_x})$ is not cohomologically trivial by \eqref{0} as $L(-2x)$ is ample (since deg$(L)\geq 3$).
Consequently, the connecting map $H^0(\mathcal{O}_{F_x})\to H^1(\Omega_{\Phi}^1|_{F_x})$ obtained from \eqref{ex1p1} is injective. Thus, \eqref{ratneed'} follows by passing to the cohomology of \eqref{ex1p1} as $H^1(\mathcal{O}_{F_x})=0$. Recall that by \claimref{cl1}, in order to show \eqref{ratneed} it is enough to prove that $$H^1(\Omega^1_{\mathbb{P}(\mathcal{E}_L)}(\log\Phi)|_{F_x}\otimes b_x^*(mL)(-2mE_x))=0\,\textrm{ for all }\, m\geq 0, x\in X.$$ This holds for $m=0$ by \eqref{ratneed'}, and for $m\geq 1$ this is a consequence of $$H^i(\Omega_{\Phi}^q|_{F_x}\otimes b_x^*(mL)(-2mE_x))=0\,\textrm{ for all }\, m\geq 1,x\in X, q=0,1$$ via \claimref{cl2}, which is easy to see as the degree of the curve is $\geq 3$ (or use \claimref{cl3}). 
\end{proof}

\begin{proof}[Proof of \corollaryref{hodgesym}]
Since $\dim(\Sigma)=3$, it is enough to show \eqref{hdb} when $0\leq p\leq 1$. This is an immediate consequence of \theoremref{prerat} and \cite[Corollary 4.1]{SVV}.
\end{proof}

\bibliography{SV.bbl}

\end{document}